\documentclass[a4paper,11pt,DIV=11,
 abstract=on,
 enabledeprecatedfontcommands
 ]{scrartcl}

\usepackage[margin=2.5cm]{geometry}
\usepackage{algorithm}
\usepackage{mathtools}
\usepackage[caption=false]{subfig}
\usepackage{graphbox}
\mathtoolsset{showonlyrefs}
\usepackage{amsmath}
\usepackage{amssymb}
\usepackage{array}
\usepackage{enumitem}
\usepackage{amsthm}
\usepackage{mathdots}
\usepackage[noend]{algpseudocode}
\usepackage{graphicx}
\usepackage{color}
\usepackage{listings}
\usepackage{hyperref}
\usepackage[utf8]{inputenc}
\usepackage[authoryear,round,longnamesfirst]{natbib}

\newtheorem{theorem}{Theorem}[section]

\newtheorem{remark}[theorem]{Remark}

\newtheorem{proposition}[theorem]{Proposition}

\newcommand{\weakly}{\rightharpoonup}\normalfont
\newcommand{\N}{\ensuremath{\mathbb{N}}}

\newcommand{\Z}{\ensuremath{\mathbb{Z}}}

\newcommand{\R}{\ensuremath{\mathbb{R}}}

\newcommand{\X}{\mathbb{X}}
\newcommand{\Y}{\mathbb{Y}}

\newcommand{\supp}{\textnormal{supp}}

\newcommand{\norm}[1]{\left\Vert #1\right\Vert}

\newcommand{\dx}{\,\mathrm{d}}
\newcommand{\tT}{\mathrm{T}}

\newcommand{\KL}{\text{KL}}

\newcommand{\Bnuxi}{B(\nu, r)} 
\newcommand{\Bmuxi}{B(\mu, r)}

\newcommand{\sighalf}{\Sigma^{\scriptscriptstyle -1/2}}
\newcommand{\pieps}{\hat\pi_\varepsilon}

\def\3{\ss}

\newmuskip\pFqmuskip

\newcommand*\pFq[6][8]{
  \begingroup 
  \pFqmuskip=#1mu\relax
  \begingroup\lccode`\~=`\,
  \lowercase{\endgroup\let~}\pFqcomma
  {}_{#2}F_{#3}{\left(\genfrac..{0pt}{}{#4}{#5};#6\right)}%
  \endgroup
}
\newcommand*\pRegFq[6][8]{
  \begingroup 
  \pFqmuskip=#1mu\relax
  \begingroup\lccode`\~=`\,
  \lowercase{\endgroup\let~}\pFqcomma
  {}_{#2}\tilde{F}_{#3}{\left(\genfrac..{0pt}{}{#4}{#5};#6\right)}%
  \endgroup
}

\newcommand\mystack[2]{\genfrac{}{}{0pt}{}{#1}{#2}}

\newcommand{\pFqcomma}{\mskip\pFqmuskip}

\DeclareMathOperator*{\id}{id}

\DeclareMathOperator*{\argmin}{argmin}

\DeclareMathOperator{\OT}{OT}

\newcommand{\PRd}{\mathcal{P}(\R^d)}

\newcommand\curwidth{0.38\textwidth}
\newcommand\curimgwidth{\textwidth}

\newcommand\curbarwidth{0.4\textwidth}
\newcommand\curboxwidth{0.4\textwidth}
\newcommand\curfolder{balanced_1e-2}

\numberwithin{equation}{section} 
\AtBeginDocument{
	\label{CorrectFirstPageLabel}
	
}

\hypersetup{pdfauthor=author}

\begin{document}
\title{Transfer Operators from Optimal Transport\\ Plans
for Coherent Set Detection}

\author{%
	{\sc
		P\'eter Koltai
	} \\[2pt]
	{\large Free University Berlin, Arnimallee 6, 14195 Berlin, Germany}\\[6pt]
	{\sc and}\\[6pt]
	{\sc
		Johannes von Lindheim\textsuperscript{*}, 
		Sebastian Neumayer}\\
	{\sc
		and 
		Gabriele Steidl
	}\\[2pt]
	{\large Technical University Berlin, Strasse des 17. Juni 135, 10623 Berlin, Germany}\\[2pt]
	{\large\textsuperscript{*}Corresponding author. Email: vonlindheim@tu-berlin.de}
}
 	\maketitle
	
	\footnotetext[1]{
		TU Berlin,
		Strasse des 17. Juni 135,
		10623 Berlin, Germany,
		\{name\}@math.tu-berlin.de.}
	\footnotetext[2]{
		FU Berlin,
		Arnimallee 6,
		14195 Berlin, Germany,
		peter.koltai@fu-berlin.de}

\begin{abstract}
The topic of this study lies in the intersection of two fields.
One is related with analyzing transport phenomena in complicated flows.
For this purpose, we use  so-called coherent sets: non-dispersing, possibly moving regions in the flow's domain.
The other is concerned with reconstructing a flow field from observations of its action on a measure, which we address by optimal transport.
We show that the framework of optimal transport is well suited for delivering the formal requirements on which a coherent-set analysis can be based on.
The necessary noise-robustness requirement of coherence can be matched by the computationally efficient concept of unbalanced regularized optimal transport.
Moreover, the applied regularization can be interpreted as an optimal way of retrieving the full dynamics given the extremely restricted information of an initial and a final particle distribution moving according to Brownian motion.
\end{abstract}





\section{Introduction} \label{sec:intro}

Over several decades, much research interest in dynamical systems has been devoted to the study of transport phenomena. 
One studied notion is that of \emph{coherence}.
Unfortunately, there is no generally-agreed-upon-definition of \emph{coherent structure}, other than the necessity for a structure to show some persistence over a longer time horizon.
One way to address the question in fluid mechanics is to consider tracer particles in the flow and study local and global geometric deformations of fluid elements attached to their trajectories under advection. This gives rise to a \emph{Lagrangian} analysis of coherent structures, see~\citet{romkedar_etal_1990, wiggins_92, haller1998finite, aref_02, jones2002invariant, wiggins2005dynamical, shadden2005definition, froyland_padberg_09, Thi12, FrPa14, KaKe20, KoRe18, HaKaKo18} for geometric and probabilistic approaches to this. 

For this study, we focus on the notion of (finite-time) \emph{coherent sets} as introduced by \citet{FrSaMo10} and \cite{froyland2013analytic}, which are defined as sets of particles that are minimally dispersive, or hard for the tracer particles to escape from during a fixed finite time interval. 
The methods detecting coherent sets, or even any coherent structures, are relying on dynamical information in the form of trajectories of individual tracer particles or quantities from which these can be obtained, e.g., velocity fields; see~\citet{hadjighasem2017critical} for an overview.

In this work, we consider the case, where the underlying flow map is not accessible and merely densities of a passive (advected) scalar are given at some initial and final time instance.
In a space-discrete version, this amounts to unordered sets of particle positions at the two time instances.
We can think of this situation as forgetting the ``identity'' or ``label'' information of individual particles, i.e., losing almost all information of the full dynamics.
Due to this, anything we can hope to identify are large and robust dynamical structures, like coherent sets. Previous works addressing the task of finding coherent structures from this kind of data are based on optical flow (more precisely, Advection Corrected Correlation Image Velocimetry) and Lagrangian coherent structure analysis~\citep{hadjighasem2016geodesic}, motion segmentation~\citep{almomani2018go}, and ensemble Kalman filter~\citep{santitissadeekorn2019iterative} for image sequences.
Typical applications include cases where the dynamics are only indirectly observed through radar or satellite images of precipitation intensity, see \citet{dwd} for an example, or other quantities.

The fields of computer vision and image processing have long been concerned with the problem of recovering dynamical information (motion) from sequences of images, for an overview see \citet{BPS2014}. 
Optical flow and motion segmentation are just two of them, which can be used to build a coherence analysis upon. Our main goal---and contribution---is to show that the theories (and state of the art computation) of \emph{optimal transport} (OT) and of coherent sets are naturally connected.
Fig.~\ref{fig:double_gyre_clustering} shows a segmentation example based on our findings, see Subsection~\ref{sec:Ex4}.
Throughout this paper, we
choose intentionally a level of detail on which this connection can be emphasized in a largely self-contained manner.
\renewcommand\curfolder{img/double_gyre}
\renewcommand\curwidth{0.35\textwidth}
\begin{figure}[tbp]
    \centering
    \subfloat[Initial time step.]
    {\includegraphics[width=\curwidth]{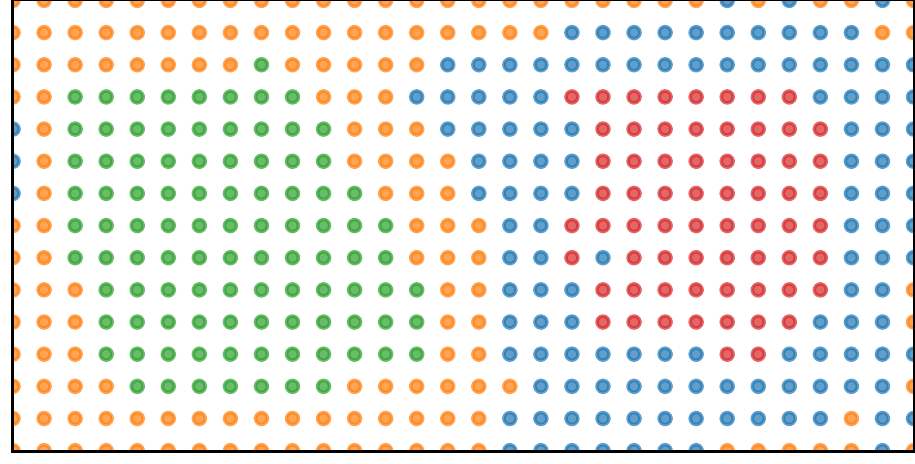}}
    \hspace*{1cm}
    \subfloat[Final time step.]
    {\includegraphics[width=\curwidth]{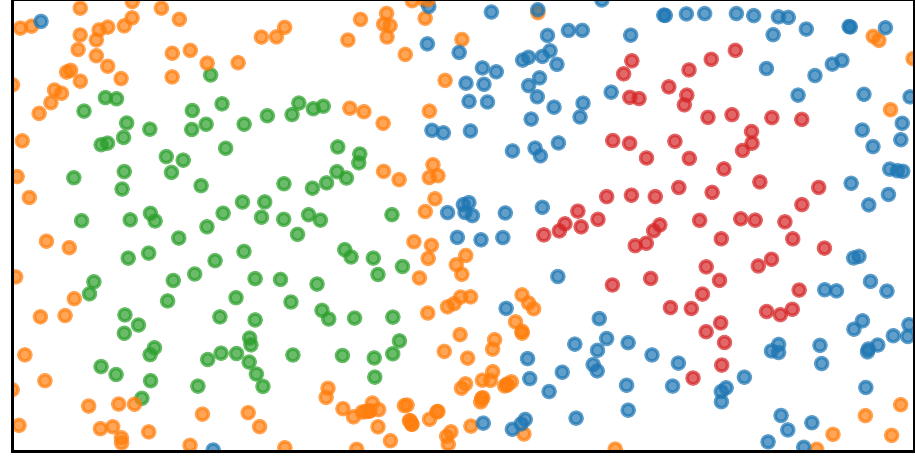}}
    \caption{Initial and end configuration of a dynamical system consisting of two counter-rotating groups of unlabeled particles.
    The whole movement of the system can be seen in the video in the supplementary material.
    The color scheme encodes a partition into ``coherent'' sets obtained with our method in Subsection~\ref{sec:Ex4}.}
    \label{fig:double_gyre_clustering}
\end{figure}

In a nutshell, we construct Frobenius--Perron operators from transport plans of (unbalanced) regularized optimal transport and use them to find coherent sets in evolving densities or particle ensembles.
Such transport plans can be interpreted as small random perturbations  of deterministic maps, naturally introducing a small amount of ``noise'' that is used for the definition of coherence~\citep{froyland2013analytic}.
Moreover, regularized optimal transport is an optimal choice in the sense that it yields a kind of most likely transport plan, given all that we know of the dynamics is how it maps a single distribution from an initial to a final time and that particles move according to Brownian motion.

The outline of this paper is as follows: 
In Section~\ref{sec:basics}, we recall basic notions from measure theory, and in  Section~\ref{sec:OT}, we introduce properties of OT, regularized OT
and unbalanced regularized OT.
The segmentation model under consideration and the corresponding optimization problem
as well as its relaxation are presented in Section~\ref{sec:segm}. 
Section~\ref{sec:frobenius_perron} deals with Frobenius--Perron transfer operators  and the relation to OT. In particular, we construct appropriate kernels
of transfer operators in two different ways, namely by i) smoothing of OT transport plans, 
and ii) by using regularized OT plans, where the smoothing is already inherent.
Both in i) and ii) the kernels converge to the OT plan if the smoothing parameter 
goes to zero.
As we point out in Section~\ref{sec:schroedinger}, the usage of regularized OT plans can be motivated from a statistical physics perspective.
To this end, we give accessible insights, namely just for two time steps, for the  Schr\"odinger question.
We outline a discrete numerical approach in Section~\ref{sec:discrete}.
Various proof-of-concept examples are presented in Section~\ref{sec:numerics}.
Finally, in Section~\ref{sec:conclusions}, conclusions are drawn.

\section{Preliminaries} \label{sec:basics}
%
Throughout this paper, let $\X,\Y \subset \mathbb R^d$ be compact sets equipped with the Borel $\sigma$-algebras $\mathcal{B}(\X)$, $\mathcal{B}(\Y)$ induced by the subspace topology, respectively.
Assume that the boundaries have Lebesgue measure (denoted by $\lambda$) zero and that the sets fulfill $\min_{x \in \X}\lambda(B_\varepsilon(x)), \min_{y \in \Y}\lambda(B_\varepsilon(y)) > 0$ for every $\varepsilon > 0$ and
\begin{equation}\label{eq:boundary}
\sup_{\varepsilon > 0} \frac{\max_{x \in \X}\lambda(B_\varepsilon(x))}{\min_{x \in \X}\lambda(B_\varepsilon(x))} < C_\X,
\qquad
\sup_{\varepsilon > 0} \frac{\max_{y \in \Y}\lambda(B_\varepsilon(y))}{\min_{y \in \Y}\lambda(B_\varepsilon(y))} < C_\Y,
\end{equation}
for some $C_\X, C_\Y >0$, 
with balls $B_\varepsilon(x) \coloneqq \{x' \in \X: \|x'-x\|_2 < \varepsilon \}$ in the appropriate spaces.
This condition holds for domains satisfying the uniform cone condition, see \citet{AdamsFournier03}, which is fulfilled if $\X$ and $\Y$ have Lipschitz boundaries.
By $\mathcal M(\X)$ we denote the linear space of all finite signed Borel measures on $\X$, 
by $\mathcal M^+(\X)$ the subset of non-negative measures, 
and by $\mathcal{P}(\X)$ the set of Borel probability measures on $\X$.
The closed set $\supp(\mu) \coloneqq \{ x \in \X: B \subset \X \text{ open, }x \in B  \implies \mu(B) >0\}$ is called the \emph{support} of a measure $\mu$.
Further, the \emph{total variation} measure of $\mu \in \mathcal M(\X)$ is defined by
\[
|\mu|(B) \coloneqq \sup \Bigl\{ \sum_{k=1}^\infty |\mu(B_k)|:
\bigcup\limits_{k=1}^\infty B_k = B, \, B_k \; \mbox{pairwise disjoint}\Bigr\}.
\]
Equipped with the norm $\| \mu\|_{\mathcal M} = |\mu|(\X)$ the space $\mathcal M(\X)$ becomes a Banach space.
By $C(\X)$ we denote the Banach space of continuous, real-valued functions on 
$\X$ with norm $\| \varphi\|_{C(\X)} \coloneqq \max_{x \in \X} |\varphi(x)|$.
The space $\mathcal M(\X)$ can be identified via Riesz' representation theorem with the dual space of $C(\X)$
and
the weak-$\ast$ topology on $\mathcal M(\X)$ gives rise to the \emph{weak convergence of measures}, i.e., a sequence $\{\mu_k\}_{k \in \N} \subset \mathcal M(\X)$ converges \emph{weakly} to $\mu$ and we write $\mu_k \weakly \mu$, if
\begin{equation}
\lim_{k \to \infty} \int_{\X} f \dx \mu_k = \int_{\X} f \dx \mu \qquad \text{for all } f \in C(\X).
\end{equation}
Note that the set $\mathcal{P}(\X)$ is weakly compact.

For a non-negative, finite measure $\mu$ and $p \in [1,\infty)$, let $L_p(\X,\mu)$ 
be the Banach space (of equivalence classes) of complex-valued functions with norm
\[\|f\|_{L_p(\X,\mu)} = \Bigl( \int_\X |f|^p \dx \mu \Bigr)^\frac1p < \infty.\]
For the Hilbert space $L_2(\X,\mu)$ we use the notation 
$\langle f,g \rangle_\mu \coloneqq \int_\X f g \dx \mu$.
By $1_A$ we denote the characteristic function of a (measurable) set $A$, defined by
\[
1_A(x)\coloneqq
\begin{cases}
1 & x\in A \\
0 & x\not\in A.
\end{cases}
\]
Let $\mathcal A \subseteq \mathcal B(\X)$ be a sub-$\sigma$-algebra. 
A mapping $g\colon \X \to \R$ is called \emph{conditional expectation}
of $f\in L_1(\X, \mu)$ if $g$ is $\mathcal A$-measurable and for all $A \in \mathcal A$ it holds
\begin{equation} \label{cond_exp}
\int_A g \dx \mu = E(1_A g) = E(1_A f) = \int_A f \dx \mu.
\end{equation}
In this case, we write $g = E(f|\mathcal A)$.

A measure $\nu \in \mathcal M(\X)$ is \emph{absolutely continuous} with respect to $\mu \in \mathcal M(\X)$, 
abbreviated by $\nu \ll \mu$, 
if for every $A \in \mathcal B(\X)$ with $\mu(A) = 0$ we have $\nu(A) = 0$.
If $\mu, \nu \in \mathcal M^+(\X)$ satisfy $\nu \ll \mu$, then the \emph{Radon--Nikodym derivative} $\tfrac{\dx \nu}{\dx \mu}= \sigma_\nu \in L_1(\X,\nu)$  exists and $\nu = \sigma_\nu \mu$.

Let $T\colon \X \to \Y$ be a measurable function, i.e., $T^{-1}(A) \in \mathcal{B}(\X)$ for all $A \in \mathcal{B}(\Y)$.
Then, the \emph{push-forward} measure of $\mu$ by $T$ is defined as $T_\# \mu  \coloneqq \mu \circ T^{-1}$.
A measurable function $f$ on $\Y$ is integrable with respect to  $\nu \coloneqq T_\# \mu$ 
if and only if the composition $f \circ T$ is integrable with respect to the measure $\mu$.
In this case, the integrals coincide, i.e., it holds
\begin{align}\label{eq:push_f}
	\int_\Y f \dx \nu &= \int_{\X} f\circ T \dx \mu.
\end{align}
For the Lebesgue measure $\lambda$, we abbreviate $\dx x$ instead of $\dx \lambda(x)$
throughout the paper.

\section{Optimal Transport and its Regularization} \label{sec:OT}
In this section, we collect results on OT and its regularized version by the  Kullback--Leibler divergence, which we couple with so-called Frobenius--Perron operators in order to segment images
in the subsequent sections.
Moreover, we describe unbalanced OT, which appears to be useful in our numerical examples.
The following discussion is based on \citet{CP2019} and \cite{S2015}.

\paragraph{Optimal Transport}
For a non-negative, symmetric and Lipschitz continuous cost function $c \in C(\X \times \Y)$ 
and given measures $\mu\in \mathcal P(\X) ,\nu \in \mathcal P(\Y)$, the \emph{Monge problem of optimal transport}  consists 
in finding a measurable function $\hat T\colon \X \to \Y$, called \emph{optimal transport map}, that realizes
 \begin{equation}
\inf_{T} \Bigl\{ \int_{\X} c\bigl(x,T(x)\bigr) \dx \mu(x) \colon \; T_{\#}\mu = \nu\Bigr\}.
\end{equation}
If a map  $T\colon \X \to \Y$ solely fulfills $T_{\#}\mu = \nu$, we call it a \emph{transport map}
between $\mu$ and $\nu$.
In contrast to Monge's problem, \emph{Kantorovich's relaxation} allows the mass to be split, i.e., it aims to find a minimizer of 
 \begin{equation}\label{Monge_Kantorovich_problem}
 \OT(\mu,\nu) \coloneqq \inf_{\pi \in\Pi(\mu,\nu)} \int_{\X \times \Y} c(x,y) \dx \pi(x,y),
\end{equation}
where $\Pi(\mu,\nu)$ 
denotes the set of all joint probability measures $\pi$ on $\X \times \Y$ with marginals $\mu$ and $\nu$. 
We refer to the measures of $\Pi(\mu,\nu)$ as \emph{transport plans}
between $\mu$ and $\nu$.
In our setting, the OT functional $\pi \mapsto  \int_{\X \times \Y} c \dx \pi$ is continuous and
\eqref{Monge_Kantorovich_problem} has a solution $\hat \pi$,  called \emph{optimal transport plan}.
Every transport map  $T\colon \X \to \Y$ between $\mu$ and $\nu$ induces a transport plan
$\pi = (\id_\X,T)_{\#}\mu \in \Pi(\mu,\nu)$, i.e.,
$$
\int_{\X \times \Y} h(x,y) \dx \pi(x,y) = \int_\X h\bigl(x,T(x)\bigr)  \dx \mu(x) \quad \text{for all} \, h \in C(\X \times \Y).
$$
Further, the \emph{$c$-transform} $\varphi^{c} \in C(\Y)$ of $\varphi \in C(\X)$ is defined as $\varphi^{c}(y) = \min_{x\in\X}\{c(x,y)-\varphi(x)\}$
and a function $\varphi^{c} \in C(\Y)$ is called \emph{$c$-concave} if it is the $c$-transform of some function $\varphi \in C(\X)$.
The dual formulation of the OT problem \eqref{Monge_Kantorovich_problem} reads
\begin{equation}\label{Wdual}
\OT(\mu,\nu) = \max_{ \substack{(\varphi,\psi)\in C(\X)\times C(\Y)\\
\varphi(x)+\psi(y)\leq c(x,y)}}
\int_{\X}\varphi \dx \mu + \int_{\Y} \psi \dx \nu.
\end{equation}
Maximizing pairs are of the form $(\varphi,\psi) = (\hat \varphi, \hat \varphi^c)$ for a $c$-concave function 
$\hat \varphi$ and fulfill $\hat \varphi(x) + \hat \varphi^c(y) = c(x,y)$ in $\supp(\hat \pi)$, where $\hat \pi$ is any optimal transport plan.
The function $\hat \varphi$ is called 
(Kantorovich) \emph{potential} for the couple $(\mu,\nu)$.
In our applications, we want to focus on settings, where an optimal transport map exists. The following theorem can be found, e.g., in \citet[Thm.~1.17]{S2015}.
                                                                             
\begin{theorem}\label{thm:tmap}
Let $\mu,\nu\in \mathcal P (\X)$, where $\mu$ is absolutely continuous 
with respect to the Lebesgue measure and let $c(x,y)=h(x-y)$ with a strictly convex function $h$.
Then, there exists a unique optimal transport plan $\hat \pi\in \Pi(\mu,\nu)$ that is induced by the optimal transport map $\hat T$.
Moreover, there exists a Kantorovich potential $\hat \phi$ which is linked to $\hat T$ via 
$\hat T(x) =x-(\nabla h)^{-1}(\nabla \hat \phi(x))$.
\end{theorem}

For $c(x,y) \coloneqq \|x-y\|_p^p$, $p \in [1,\infty)$, the optimal transport cost induces the \emph{$p$-Wasserstein distance}  
\begin{align} \label{eq:OTprimal}
W_p(\mu,\nu) \coloneqq \OT(\mu,\nu)^\frac{1}{p} 
= \biggl( \min_{\pi \in \Pi(\mu,\nu)} 
\int_{\X^2} \| x- y\|_p^p \mathrm{d} \pi(x,y) \biggr)^\frac{1}{p},
\end{align}
which metrizes the weak topology on $\mathcal P(\X)$.
Indeed, due to boundedness of $\X$, we have that $\mu_k \weakly \mu$
if and only if $\lim_{k \to \infty} W_p (\mu_k, \mu) = 0$.
In our numerical examples, we use the Wasserstein-2 distance. 
By Theorem~\ref{thm:tmap}, we have in particular for 
$c(x,y) \coloneqq \tfrac{1}{2}\norm{x-y}_{2}^{2}$
that
$\hat T(x)=x-\nabla \hat \phi(x)$.

\paragraph{Regularized Optimal Transport}
We recall that the \emph{Kullback--Leibler divergence} $\mathrm{KL}\colon{\mathcal M^+}(\X) \times {\mathcal M^+}(\X) \rightarrow \mathbb [0, +\infty]$ is defined for $\mu \ll \nu$ by
\begin{equation} \label{KLdef}
\mathrm{KL} (\mu,\nu) \coloneqq 
\int_{\X} \log\Bigl(\frac{\dx \mu}{\dx \nu} \Bigr) \, \dx \mu + \nu(\X) - \mu(\X) ,
\end{equation}
and by $\mathrm{KL} (\mu,\nu) \coloneqq + \infty$ otherwise.
For $\mu,\nu \in \mathcal{P}(\X)$ the last two summands in \eqref{KLdef} cancel each other.
The \emph{Kullback--Leibler divergence} is strictly convex and weakly lower-semicontinuous with respect to the first variable.
For $ \mu \in {\mathcal P}(\X)$, $ \nu  \in \mathcal P(\Y)$ and $\varepsilon > 0$, the \emph{regularized OT} problem is defined as
\begin{align}
\OT_{\varepsilon}(\mu,\nu) 
\coloneqq&  
\min_{\pi \in \Pi(\mu,\nu)} 
\,
\Big\{
\int_{\X \times \Y}  c\, \mathrm{d} \pi
+ \varepsilon \mathrm{KL} (\pi,\mu \otimes \nu) \Big\}\label{sinkhorn}\\
=&\; \varepsilon
\min_{\pi \in \Pi(\mu,\nu)} 
\, \mathrm{KL} (\pi, \exp(-c/\varepsilon) \mu \otimes \nu) - \varepsilon \int_{\X \times \Y}  \exp(-c/\varepsilon) -1 \dx (\mu \otimes \nu).\label{sinkhorn_kernel} 
\end{align}
A dual formulation is given in the next theorem, see \citet{CP2019,neumayer2020optimal}.

\begin{theorem}\label{prop:dual}
	The (pre-)dual problem of $\OT_\varepsilon$ is given by
	\begin{align}
	\max_{(\varphi,\psi)\in C(\X) \times C(\Y)}
	\Big\{ \int_{\X}\varphi \dx \mu + \int_{\Y} \psi \dx \nu
	- \varepsilon \int_{\X \times \Y} \exp\Bigl(\frac{\varphi(x) + \psi(y) - c(x,y)}{\varepsilon}\Bigr) -1
	\dx (\mu \otimes \nu) \Big\}.\label{pre-dual}
	\end{align}
	The optimal potentials $\hat \varphi_\varepsilon \in C(\X)$, $\hat \psi_\varepsilon \in C(\Y)$ 
	exist and are unique on $\supp(\mu)$ and $\supp(\nu)$, respectively (up to an additive constant).
	They are related to the optimal transport plan $\hat \pi_\varepsilon$ by
	\begin{equation}\label{eq:PDrelation}
	\hat\pi_\varepsilon = \exp\Bigl(\frac{\hat \varphi_\varepsilon(x) + \hat \psi_\varepsilon(y) - c(x,y)}{\varepsilon}\Bigr) \mu \otimes \nu.
	\end{equation}
\end{theorem}
Under certain assumptions, see Proposition \ref{prop:conv-reg}, the regularized OT plans $\hat \pi_\varepsilon$ converge weakly to the  OT plan $\hat \pi$ as $\varepsilon$ goes to zero. 
The corresponding discrete optimization problem can be solved efficiently using Sinkhorn's algorithm, see \citet{C2013,FSVATP2019} for more details. 

\paragraph{Unbalanced Optimal Transport}

In practical applications, we often have to deal with noisy data.
In this case, the computation of a (regularized) optimal transport plan is often unreasonable. e.g., if $\mu, \nu$ are positive measures such that $\mu(\X)\neq \nu(\Y)$.
To resolve this issue, we can use unbalanced regularized optimal transport, which allows small deviations between the input measures and the marginals of the associated transport plan.
This approach corresponds to the optimization problem
\begin{equation}\label{eq:unbalanced_entropic_ot}
\min_{\pi\in \mathcal{M}^+(\X \times \Y)} \left\{ \int_{\X\times \Y} c\dx \pi + \varepsilon \KL(\pi\vert \mu\otimes\nu) 
+ \kappa \bigl(\KL(\pi_0 \vert \mu) + \KL(\pi_1\vert \nu)\bigr)\right\},
\end{equation}
where $\pi_0$ and $\pi_1$ denote the marginals with respect to the first and second component, respectively, see \citet{CPSV18}.
Similarly as in the balanced case, we can rewrite the problem in dual form
\begin{align}
\max_{(\varphi,\psi)\in C(\X) \times C(\Y)} \left\{\vphantom{\int_\X}\right.&\kappa \biggl(\int_\X 1 - \exp\bigl(-\tfrac{\varphi}{\kappa}\bigr) \dx \mu +\int_\Y 1- \exp\bigl(-\tfrac{\psi}{\kappa}\bigr)\dx\nu\biggr)\\ &- \left.\varepsilon\int_{\X\times\Y} \exp\bigl(\tfrac{\varphi(x) + \psi(y) -c(x, y)}{\varepsilon}\bigr) - 1 \dx (\mu \otimes \nu)\right\}
\end{align}
and the optimal solution $\hat \pi_{\varepsilon,\kappa}$ of the primal problem is related to the optimal solutions $\hat \varphi_{\varepsilon,\kappa}, \hat \psi_{\varepsilon,\kappa}$ of the dual problem via
\begin{equation} \label{eq:ot-unbalanced}
\hat \pi_{\varepsilon,\kappa} = \exp\bigl(\tfrac{\hat \varphi_{\varepsilon,\kappa}(x) + \hat \psi_{\varepsilon,\kappa}(y) -c(x,y)}{\varepsilon}\bigr)\mu \otimes \nu.
\end{equation}
Further, the formal limit $\kappa \to \infty$ results in the original balanced problem.
Note that the OT plan $\hat \pi_{\varepsilon,\kappa}$ does not have marginals $\mu$ and $\nu$.
However, if $\kappa$ is large, the marginals $\mu_\kappa$ and $\nu_\kappa$ of $\hat \pi_{\varepsilon,\kappa}$ 
are close to the original input measures and in some sense we can interpret them as smoothed versions of those.
Similarly as for regularized OT, we can use a variant of the Sinkhorn algorithm to efficiently solve the corresponding discrete problem, see \citet{SFVTP19}.

\section{Segmentation Model}\label{sec:segm}
%
Following the lines of \citet{froyland2013analytic}, we introduce the segmentation model that we want to apply for coherent structure detection.
More precisely, we use the concept of \emph{transfer operators}, which are rigorously introduced in Section~\ref{sec:frobenius_perron}.
It is sufficient to note that the transfer operator $L$ below is the adjoint of the pullback operator $f\mapsto f\circ T$ as an operator between the respective $L^p$ spaces.
Further, it is a functional extension of the push-forward operator $T_{\#}$ from Section~\ref{sec:basics}.

Assume that we are given full information about some dynamical system by means of its linear, bounded transfer operator 
$L\colon L_2(\X,\mu) \to  L_2(\Y,\nu)$. Let $\dot\cup$ denote the disjoint union of sets.
We aim to find measurable partitions $\X= X_1\dot\cup X_2$ and $\Y= Y_1\dot\cup Y_2$ so that it holds for $k=1,2$ that
\begin{itemize}
\item $L 1_{X_k} = 1_{Y_k}$ (\emph{coherence}) and
\item $\mu(X_k) = \nu(Y_k)$ (\emph{mass conservation}). 
\end{itemize}
Note that the first condition readily implies $L 1_{\X} = 1_{\Y}$.
Based on these conditions, a natural model in order to find the partitions would be
	\begin{align}\label{eq:coh_set_problem}
	\max_{\substack{X_1\dot\cup X_2=\X,\, Y_1\dot\cup Y_2=\Y}}   \biggl\{\frac{\langle L 1_{X_1}, 1_{Y_1} \rangle_\nu}{\mu(X_1)} 
	+ \frac{\langle L 1_{X_2}, 1_{Y_2}\rangle_\nu   }{\mu(X_2)} \biggr\},
	\end{align}
	trying to optimize the alignment of $L 1_{X_k}$ and $1_{Y_k}$ in $\Y$, $k=1, 2$. The expression $\langle L \frac{1_{X_k}}{\mu(X_k)}, 1_{Y_1} \rangle_\nu$ quantifies the probability that a $\mu$-distributed random initial condition, being in $X_k$, is mapped by the dynamics into $Y_k$.
	Using standard arguments, see \citet{froyland2013analytic} or also \citet{Luxburg2007} in connection with 
	 graph cuts, problem~\eqref{eq:coh_set_problem} can be rewritten as 
	\begin{align}\label{eq:coh_set_problem_1}
	1 + \max_{\substack{X_1\dot\cup X_2 = \X,\, Y_1\dot\cup Y_2=\Y}} \langle L\psi_{X_1, X_2}, \psi_{Y_1, Y_2} \rangle_\nu,
	\end{align}
	where
	$\psi_{X_1, X_2} \coloneqq \sqrt{\tfrac{\mu(X_2)}{\mu(X_1)}} 1_{X_1} - \sqrt{\tfrac{\mu(X_1)}{\mu(X_2)}} 1_{X_2}$ and 
	$\psi_{Y_1, Y_2} \coloneqq \sqrt{\tfrac{\nu(Y_2)}{\nu(Y_1)}} 1_{Y_1} - \sqrt{\tfrac{\nu(Y_1)}{\nu(Y_2)}} 1_{Y_2}$.
	These functions fulfill
	\[
	\|\psi_{X_1, X_2}\|_{L_2(\X,\mu)} = \mu (\X)^\frac12, \quad
	\|\psi_{Y_1, Y_2}\|_{L_2(\Y,\nu)} = \nu (\Y)^\frac12, \quad
	\langle \psi_{X_1, X_2}, 1_\X \rangle_\mu = \langle \psi_{Y_1, Y_2}, 1_\Y \rangle_\nu = 0.
	\]
	Clearly, if we could find $\psi_{X_1, X_2}$ with corresponding $\psi_{Y_1, Y_2}$, the desired partition can be obtained by thresholding these functions at zero.
	Unfortunately, problem \eqref{eq:coh_set_problem} or equivalently \eqref{eq:coh_set_problem_1} is hard to solve, 
	indeed NP-hard in its finite dimensional variant.
	Instead, we relax the problem by replacing 
	$\psi_{X_1, X_2}$ with $f/\|f\|_{L_2(\X,\mu)}  \in L_2(\X,\mu)$ and 
	$\psi_{Y_1, Y_2}$ with $g/\|g\|_{L_2(\Y,\nu)}  \in L_2(\X,\nu)$.
	Then, the relaxed problem 
	\begin{align} \label{problem}
	 \max_{
	\substack{
	(f,g) \in L_2(\X,\mu) \times L_2(\Y,\nu)
	} }
	\biggl\{ 
	\frac{ \langle L f, g \rangle_\nu}{ \Vert f\Vert_{L_2(\X,\mu)} \Vert g \Vert_{L_2(\Y,\nu)} }: 
	\langle f, 1_\X \rangle_\mu = \langle g, 1_\Y \rangle_\nu = 0 
	\biggr\}
	\end{align}
allows for fuzzy instead of hard assignments.
In the following, we set ${\mathcal L} \coloneqq L^* L$ and assume that 
\begin{enumerate}[label=(A\arabic*)]
\item\label{enum:op_marg} $L 1_\X = 1_\Y$;
\item\label{enum:adjoint_marg} $\mathcal L 1_\X = \|L\|^2 1_\X$, which by \ref{enum:op_marg} is equivalent to $L^* 1_\Y = \|L\|^2 1_\X$, where we abbreviate
$\|L\| \coloneqq \|L\|_{L_2(\X,\mu)\to L_2(\Y,\nu)}$ for the operator norm;
\item\label{enum:operator_cpt_simple} $L$ is compact and the largest eigenvalue $\lambda_1 = \|L\|^2$ of ${\mathcal L} = L^* L$ is simple.
\end{enumerate}
Then, problem \eqref{problem} is equivalent to finding the second largest singular value of $L$
and the corresponding left and right singular vectors. This can be seen as follows:
The compact, self-adjoint, positive operator $\mathcal L$
has a non-negative spectrum, where the countable set of eigenvalues 
$\lambda_1 > \lambda_2 \ge \ldots \ge 0$ of $\mathcal L$
fulfills the Courant's minimax principle \citep[Thm.~4, p.~212]{BS1986},
\[
\lambda_k = \min_{V: \mathrm{codim} V \le k-1} \max_{0\not = f \in V} 
\frac{\langle \mathcal{L} f,f \rangle_\mu}{\langle f,f \rangle_\mu} =\min_{V: \mathrm{codim} V \le k-1} \max_{0\not = f \in V} \frac{ \| L f\|^2_{L_2(\Y,\nu)} }{ \|f\|^2_{L_2(\X,\mu)} }, \quad k \in \mathbb N.
\]
Using the definition of the norm, this can be rewritten as
\begin{equation}\label{courant}
\sigma_k \coloneqq \lambda_k^{\frac12} 
= 
\min_{V: \mathrm{codim} V \le k-1} \max_{ \mystack{0\not = f \in V}{0 \not = g \in L_2(\Y,\nu)} }
\frac{ \langle L f,g \rangle_\nu }{ \|f\|_{L_2(\X,\mu)} \|g\|_{L_2(\Y,\nu)}}, \quad k \in \mathbb N.
\end{equation}
The values $\sigma_k$ are the non-zero singular values of $L$, and the maximizing functions $f$ and $g$ are so-called 
left and right singular vectors of $L$ belonging to $\sigma_k$,
respectively. Note that the left singular vectors are the eigenvectors of $\mathcal L$ and 
right singular vectors are given by $g = Lf$.
By assumptions~\ref{enum:adjoint_marg} and \ref{enum:operator_cpt_simple}, $f = 1_\X$ is an eigenvector of $\mathcal L$ belonging to the simple, largest eigenvalue $\lambda_1$
so that the eigenvectors belonging to smaller eigenvalues are perpendicular to~$1_\X$.
Thus,
 \eqref{courant} becomes
\begin{equation}\label{courant_1}
\sigma_2 
= 
\max_{ \mystack{0\neq f \in L_2(\X, \mu)}{0 \neq g \in L_2(\Y,\nu)} }
\biggl\{ \frac{ \langle L f,g \rangle_\nu }{ \|f\|_{L_2(\X,\mu)} \|g\|_{L_2(\Y,\nu)} }:
\langle f, 1_\X \rangle_\mu = \langle g, 1_\Y \rangle_\nu = 0 \biggr\},
\end{equation}
where the orthogonality condition on $g$ can be added since $f\perp 1_\X$ implies $Lf\perp 1_\Y$:
To summarize, if $L\colon L_2(\X,\mu) \to  L_2(\Y,\nu)$
fulfills
the assumptions~\ref{enum:op_marg}--\ref{enum:operator_cpt_simple}, then model~\eqref{problem} can be solved by finding the left and right singular functions of $L$ belonging to the second largest singular value of $L$.
Later, we focus on compact operators arising from non-negative kernels $K\in L_2(\X\times \Y,\mu\otimes\nu)$ via
$$
(L f)(y) \coloneqq \int_\X f(x) K(x, y) d\mu (x), \quad (L^* g)(x) \coloneqq \int_\Y g(y) K(x, y) d\nu (y).
$$
Then, \ref{enum:op_marg}--\ref{enum:operator_cpt_simple} are met if the following corresponding assumptions hold for $K$:
\begin{enumerate}[label=(K\arabic*)]
		\item\label{enum:nu_marg} $\int_\X K(x, \cdot) d\mu(x) = 1_\Y\quad\nu$-a.e.,
		\item \label{enum:mu_marg}$\int_\Y K(\cdot, y) d\nu(y) = \|L\|^2 \,  1_\X\quad \mu$-a.e.
		\item \label{enum:cpt_simple} The largest eigenvalue of $L^*L$ is simple.
\end{enumerate}
\begin{remark}[Relation to graph cut segmentation]
Partitioning data using the eigenvector corresponding to the second largest eigenvalue of the so-called graph Laplacian operator, which is self-adjoint for undirected graphs,
is a frequently applied method, e.g., in image processing.
As illustrated by Fig.~\ref{fig:double_gyre_clustering} in the introduction, our approach is related to dynamical data sets,
and we aim to partition the sets simultaneously at different times using the singular vector pairs
of a not necessarily self-adjoint transfer operator, whose construction is addressed in the next sections.
\end{remark}	
\section{Frobenius--Perron Operators}\label{sec:frobenius_perron}
In this section, we consider the so-called Frobenius--Perron operator $L=P_T$ of a measurable map $T\colon \X \to \Y$.
In Subsection  \ref{sec:frobenius_perron_1}, we will see that this transfer operator $L$ 
fulfills our assumptions \ref{enum:op_marg} and \ref{enum:adjoint_marg} provided that the operator $T$
is a transport map between measures $\mu \in \mathcal P(\X)$ and $\nu \in \mathcal P(\Y)$.
Unfortunately, the resulting transfer operator will be useless for our segmentation task since it has only singular values $1$ and thus it violates \ref{enum:operator_cpt_simple}.
As a possible solution, we provide a kernel-based definition of the Frobenius--Perron operator and propose two approaches for smoothing the transport maps.
The first technique is suggested in Subsection \ref{sec:frobenius_perron_2}.
It adds a small random perturbation to the kernel of the operator, an idea that goes back to \citet{froyland2013analytic}.
The second approach, which we will favor in our numerical examples, uses regularized optimal transport plans.

\subsection{Frobenius--Perron Operators and Transport Maps}\label{sec:frobenius_perron_1}
For completeness, we introduce the Frobenius--Perron operator between two Lebesgue spaces $L^p(\X,\mu)$ and~$L^p(\Y,\nu)$. Analogous results for $\X=\Y$ and $\mu = \nu$ can be found, e.g., in~\citet[Chap.~4]{boyarsky1997laws} or~\citet{brin2002introduction,LasotaAndrzej1994CFaN}.

Assume that $T\colon \X \to \Y$ is \emph{non-singular} with respect to $\mu \in \mathcal{P}(\X)$  and $\nu\in \mathcal{P}(\Y)$, 
namely that $\nu(A) = 0$ implies $\mu\bigl(T^{-1} (A)\bigr) = 0$ for all $A \in \mathcal{B}(\Y)$.
It is immediately clear that a transport map $T$ is non-singular with respect to $\mu$ and $\nu$.
For a non-singular, measurable map $T\colon \X \to \Y$, the linear operator
$P_T\colon L_1(\X,\mu) \to L_1(\Y,\nu)$ 
is called  \emph{Frobenius--Perron operator} of $T$ if it satisfies
\begin{equation}\label{eq:PerFro}
\int_A P_T \psi \dx \nu = \int_{T^{-1} (A)} \psi \dx \mu   \quad \mathrm{for \; all} \quad A \in \mathcal{B}(\Y) .
\end{equation}
In other words, the function $P_T \psi \in L_1(\Y, \nu)$ is given for each $\psi \in L_1(\X,\mu)$ as the density of $\smash{T_{\#}(\psi\mu) = \int_{T^{-1}(\cdot)}}\psi\dx \mu$ with respect to $\nu$, which exists because
\[
\nu(A) = 0 \implies \mu(T^{-1}(A)) = 0 \implies (T_{\#}\psi\mu)(A) = 0 \quad \mathrm{for \; all} \quad A \in \mathcal{B}(\Y),
\]
i.e., $T_{\#}(\psi\mu) \ll \nu$.
In particular, integrals are preserved $\int_\Y P_T \psi \dx \nu = \int_{\X} \psi \dx \mu$.
Note that often the Frobenius--Perron operator is defined as an operator just mapping from $L_1(\X,\mu)$ to itself, see \citet{LasotaAndrzej1994CFaN}.
As $T$ is non-singular, the linear operator
$U_T \colon L_\infty(\Y,\nu) \to L_\infty(\X, \mu)$ given by
\begin{equation} \label{eq:koopmann}
U_T \psi = \psi \circ T
\end{equation}
is also well-defined. 
It is known as \emph{Koopman operator} and the adjoint of $P_T$.

\begin{remark} 
\citet{froyland2013analytic} defined the Frobenius--Perron operator $P\colon L_1(\X, \lambda) \to L_1(\Y, \lambda)$ 
for a map $T$ that is non-singular
with respect to the Lebesgue measure, i.e.,
\begin{equation}\label{eq:PerFro_leb}
\int_A (P \psi) (y) \dx y
= \int_{T^{-1} (A)} \psi(x) \dx x
\quad \mathrm{for \; all} \quad A \in \mathcal{B}(\Y).
\end{equation}
For $\mu$ and $\nu$ being absolutely continuous with respect to the Lebesgue measure
with densities $\sigma_\mu$ and $\sigma_\nu$, respectively, a transfer operator $L\colon L_1(\X, \mu) \to L_1(\Y, \nu)$ was determined by $L\psi \coloneqq P(\psi \sigma_\mu)/\sigma_\nu$. 
This implies for all $A \in \mathcal{B}(\Y)$ that
$$
\int_A L \psi (y)  \dx \nu(y)  = \int_A P( \psi \sigma_\mu)(y)  \dx y
= 
\int_{T^{-1}(A)}  \psi(x) \sigma_\mu(x) \dx x
= \int_{T^{-1}(A)}  \psi(x) \dx \mu(x).
$$
Comparing this with \eqref{eq:PerFro}, we see that for this special case our Frobenius--Perron operator coincides with the above transfer operator 
$P_T = L$.
\end{remark}

If $T$ is a transport map between $\mu$ and $\nu$, the following proposition shows that $P_T$ has the properties \ref{enum:op_marg} and \ref{enum:adjoint_marg},
but unfortunately only singular values $1$, thus violating \ref{enum:operator_cpt_simple}.

\begin{proposition}\label{lem:id}
Let a measurable map $T\colon \X \to \Y$ fulfill $\nu = T_\# \mu$ and let $P_T$, $U_T$ be defined by \eqref{eq:PerFro} and \eqref{eq:koopmann}, respectively.
Then, the following holds true:
\begin{itemize}
	\item[i)] For any $\psi \in L_1(\X,\mu)$ we have $U_T \circ P_T (\psi) = E(\psi|T^{-1}(\mathcal B(\Y)))$, which is the conditional expectation of $\psi$ with respect to the $\sigma$-algebra $T^{-1}(\mathcal B(\Y))$.
	If $T\colon \X \to \Y$ maps Borel sets of $\X$ to those of $\Y$ and is also $\mu$-essentially injective, i.e., there exists a measurable set $B$ with $\mu(B) = 1$ on which $T$ is injective, this simplifies to $U_T \circ P_T = \text{Id}$.
	\item[ii)] For any $p \in [1, \infty]$, the Frobenius--Perron operator, as well as its adjoint, 
	can be restricted to $P_T\colon L_p(\X,\mu) \to L_p(\Y,\nu)$ and $U_T\colon L_p(\Y,\nu)\to L_p(\X,\mu)$, respectively.
	\item[iii)] The operators satisfy $P_T 1_\X = 1_\Y$ and $U_T 1_\Y = 1_\X$.
\end{itemize}
\end{proposition}
\begin{proof}
i) Recall the definition of conditional expectations in \eqref{cond_exp}.
We have for any $A \in \mathcal B(\Y)$ that
\begin{align*}
\int_{T^{-1}(A)} \!\psi \dx \mu {\stackrel{\eqref{eq:PerFro}}{=}} \int_A P_T \psi \dx \nu  
{\stackrel{\eqref{eq:push_f}}{=}} \int_{\X} (1_{A} P_T \psi)\circ T \dx \mu= \int_{T^{-1}(A)} (P_T \psi)\circ T \dx \mu 
=  \int_{T^{-1}(A)} \! U_T \circ P_T (\psi) \dx \mu,
\end{align*}
which implies the first claim.
If the conditions for the second part are fulfilled, we can choose $A = T(A' \cap B)$ for any $A' \in \mathcal B(\X)$.
Then, the second claim follows from
\[\int_{A'} \psi \dx \mu = \int_{A' \cap B} \psi \dx \mu = \int_{A'\cap B} U_T \circ P_T (\psi) \dx \mu
= \int_{A'} U_T \circ P_T (\psi) \dx \mu.\]

ii)
Using part i) and the Jensen inequality for conditional expectations, it holds for any $p \in [1, \infty)$ that
\begin{align}
\int_\Y \vert P_T\psi\vert^p \dx \nu &{\stackrel{\eqref{eq:push_f}}{=}} \int_\X \vert U_T \circ P_T(\psi) \vert^p \dx \mu = \int_\X \bigl\vert E\bigl(\psi|T^{-1}(\mathcal B(\Y))\bigr)\bigr\vert^p \dx \mu\\
&\leq \int_\X E\bigl(\vert \psi \vert^p |T^{-1}(\mathcal B(\Y))\bigr) \dx \mu = \int_\X \vert \psi\vert^p \dx \mu < \infty
\end{align}
and
\begin{align}
\int_\X \vert U_T\psi\vert^p \dx \mu &= \int_\X \vert \psi \circ T \vert^p \dx \mu  = \int_\Y \vert \psi\vert^p \dx \nu < \infty.
\end{align}
For $p = \infty$, we only need to consider $P_T$ and estimate
\begin{align}
\Vert P_T\psi\Vert_{L_\infty(\Y,\nu)} = \Vert (P_T\psi)\circ T\Vert_{L_\infty(\X,\mu)} = \bigl\Vert  E\bigl(\psi|T^{-1}(\mathcal B(\Y))\bigr) \bigr\Vert_{L_\infty(\X,\mu)} \leq \Vert \psi \Vert_{L_\infty(\X,\mu)} < \infty.
\end{align}

iii) First, we have $\int_A P_T 1_\X \dx \nu = \mu\bigl(T^{-1}(A)\bigr) = \nu(A)$, which readily implies $P_T 1_\X = 1_\Y$.
Second, we obtain by definition that $U_T 1_\Y = 1_\Y \circ T = 1_\X$.
\end{proof}

By Proposition~\ref{lem:id} i), we conclude that $L=P_T$ is not suited for our segmentation problem: Since $U_T\circ P_T(1_A) = 1_A$ with $A\in T^{-1}(\mathcal B (\Y))$, the leading eigenvalue $1$ of $L^*L$ is in general not simple.
Additionally, this also implies that neither $P_T$ nor $U_T$ can be compact as otherwise the identity operator would be compact.
A possible remedy is to smooth the operator~$P_T$.
To this end, we need a kernel representation of $P_T$.
Using the family of probability measures
$(\pi_x)_{x \in \X}$ with $\pi_x\in \mathcal P(\Y)$ given by $\pi_x = \delta_{T(x)}$
for all $x\in \X$, condition \eqref{eq:PerFro} can be rewritten as
\begin{equation}\label{kernel_def}
\int_A P_T \psi \dx \nu = \int_{\X} \pi_x(A)\psi(x) \dx \mu(x) \quad \mathrm{for \; all} \quad A \in \mathcal B(\Y).
\end{equation}
Note that $(\pi_x)_{x\in \X}$ are the disintegrations of the transport plan $\pi$ that is induced by the transport map $T$ with respect to $\mu$.
Here, the measure-valued function $x\mapsto \pi_x$ is usually called the \emph{transition function}.

This can be generalized to non-deterministic systems by choosing other families of probability measures.
Assuming that each probability measure $\pi_x$ has the density $K(x,\cdot)$ w.r.t.~$\nu$ with $K\in L_1(\X \times \Y,\mu \otimes \nu)$,
we obtain for all $A \in \mathcal{B}(\Y)$ by \eqref{kernel_def} and Fubini's theorem
\begin{align*}
\int_A (P_K \psi) (y) \dx \nu(y) = \int_{\X} \int_A K(x,y) \psi(x) \dx \nu(y) \dx \mu(x) = \int_A \int_{\X} K(x,y) \psi(x) \dx \mu(x) \dx \nu(y).
\end{align*}
Hence, we conclude that
\begin{equation}
	(P_K \psi) (y) = \int_{\X} K(x,y) \psi(x) \dx \mu(x)
\end{equation}
holds $\nu$-a.e., which corresponds to the definition of the Frobenius--Perron operator for kernels, see also \citet{KNKW18}.
The corresponding adjoint operator $U_K \colon L_\infty(\Y,\nu) \to L_\infty(\X, \mu)$ is given by
\[(U_K \psi) (x) = \int_{\Y} K(x,y) \psi(y) \dx \nu(y).\]
In this framework, we call $K\in L_1(\X \times \Y,\mu \otimes \nu)$ a \emph{transition kernel with reference measures $\mu$ and $\nu$}, to which $P_K$ as defined above is associated.

There is also a dynamical reason for the introduction of the transition density $K(x, \cdot)$.
In the following, we will view them as $T(x)$ with a small random perturbation.
As shown by \citet{froyland2013analytic}, segmentation of such perturbed dynamics yields partitions $\X=X_1\dot\cup X_2$ and $\Y=Y_1\dot\cup Y_2$ that are robust to small noise.

\subsection{Smoothed Transition Kernels}\label{sec:frobenius_perron_2}
In the previous subsection, we have seen that the operators $P_T$ arising from transport maps fulfill
only some of the required properties to serve as operators in our segmentation task.
Next, we are interested in blurred kernels that fulfill \ref{enum:nu_marg} and \ref{enum:mu_marg} and have a nontrivial spectrum.
Compared to the approach proposed by \citet{froyland2013analytic}, we prefer to avoid the domain padding 
and adapt the weights in the kernels instead.
This leads to modifications of the corresponding proofs.

For $\varepsilon > 0$ and 
a measurable map $T\colon \X \to \Y$ with $\nu = T_\# \mu$,  we define $\eta_{\varepsilon,x} \coloneqq 1_{B_\varepsilon(x)}/\lambda(B_\varepsilon(x))$
and $k_\varepsilon\colon \X \times \Y \to \R$ by
\begin{align}
k_\varepsilon (x,y) &\coloneqq \int_\X \eta_{\varepsilon,x}(z) \eta_{\varepsilon,T(z)}(y) \dx z
= 
\frac{1}{\lambda(B_\varepsilon(x))} \int_{B_\varepsilon(x) \cap T^{-1} \left(B_\varepsilon(y) \right)} 
\frac{1}{\lambda\left(B_\varepsilon \left(T(z) \right)\right)} \, \dx z.
\end{align}
Due to \eqref{eq:boundary}, $k_\varepsilon$ is bounded from above by
$C_{\varepsilon} \coloneqq 1/\min_{y \in \Y} \lambda(B_\varepsilon(y)) = C_\Y / \max_{y \in \Y} \lambda(B_\varepsilon(y))$.
Hence, $\sigma_\varepsilon (y) \coloneqq \int_\X k_\varepsilon(x,y) \dx \mu(x)$ is well-defined and bounded and we can introduce a new measure $\nu_\varepsilon$ via $\nu_\varepsilon \coloneqq \sigma_\varepsilon \lambda$.
Finally, our smoothed kernel $K_\varepsilon\colon \X \times \Y \to \R$
reads as
\begin{equation} \label{eq:kernel_smoothed}
K_{\varepsilon}(x,y) 
\coloneqq 
\frac{k_\varepsilon(x,y)}{\sigma_\varepsilon(y)}.
\end{equation}

The following proposition shows that the
operator $L_\varepsilon\colon L_2(\X,\mu) \to L_2(\Y,\nu_{\varepsilon})$ defined by
\begin{equation} \label{FP_kernel}
L_\varepsilon \psi (y) \coloneqq \int_\X K_\varepsilon (x,y) \psi(x) \dx \mu(x)
\end{equation}
is suited for our segmentation model \eqref{problem}.

\begin{proposition} \label{lem:prop_kern}
The kernel $K_\varepsilon\colon \X \times \Y \to \R$ defined in \eqref{eq:kernel_smoothed} is double-stochastic, i.e., $K_\varepsilon \in L_2(\X \times \Y, \mu \otimes \nu_\varepsilon)$ is non-negative and
\begin{align}
    \int\limits_\X K_{\varepsilon}(x,\cdot)  \dx \mu(x) = 1_\Y \,\,\,\nu_\varepsilon\text{-a.e.}, \quad
    \int\limits_\Y K_{\varepsilon}(\cdot,y)  \dx \nu_\varepsilon(y) = 1_\X \,\,\,\mu\text{-a.e.}
\end{align}
Further, $L_\varepsilon$ fulfills the properties \ref{enum:nu_marg} and \ref{enum:mu_marg}.
\end{proposition}
\begin{proof}
The first marginal property follows directly by definition of $K_\varepsilon$.
Further, the second one can be verified by
\begin{align*}
\int\limits_\Y K_{\varepsilon} (x,y)  \dx \nu_\varepsilon(y) 
= 
\int_\Y \int_\X \eta_{\varepsilon,x} (z) \eta_{\varepsilon,T(z)} (y) \dx z \dx y = 
\int_\X \eta_{\varepsilon,x} (z) \int_\Y \eta_{\varepsilon,T(z)} (y) \dx y \dx z 
= 
\int_\X \eta_{\varepsilon,x} (z)  \dx z = 1.
\end{align*}
Further, non-negativity of $K_\varepsilon$ follows directly by definition, and square-integrability can be checked as follows
\begin{align*}
\int_\X \int_\Y K_\varepsilon(x,y)^2 \dx \nu_\varepsilon(y) \dx \mu(x)
&=
\int_\X \int_\Y \frac{k_\varepsilon (x,y)^2}{\sigma_\varepsilon(y)^2} \sigma_\varepsilon(y) \dx y \dx \mu(x)=
\int_\X \int_\Y \frac{k_\varepsilon (x,y)^2}{\sigma_\varepsilon(y)} \dx y \dx \mu(x)\\
&= 
\int_\Y \frac{\int_\X k_\varepsilon (x,y)^2 \dx \mu(x)}{\int_\X k_\varepsilon (x,y) \dx \mu(x)} \, \dx y \le
\int_\Y C_\varepsilon \frac{\int_\X k_\varepsilon (x,y) \dx \mu(x)}{\int_\X k_\varepsilon (x,y) \dx \mu(x)}\, \dx y = C_\varepsilon \lambda(\Y).
\end{align*}
Finally, \cite[Thm.~6.18]{Folland1984} implies $\Vert L_\varepsilon \Vert = 1$ and hence $L_\varepsilon$ fulfills the properties \ref{enum:nu_marg} and \ref{enum:mu_marg}.
\end{proof}
\begin{remark}\label{rem:simplicity}
Simplicity of the largest eigenvalue of $L_\varepsilon^*L_\varepsilon$ can be shown in a similar way as in \citet[Prop.~3]{froyland2013analytic}, e.g., for $\X=\Y$ being a connected domain and $T$ being a diffeomorphism with Jacobi determinant uniformly bounded from above and below.
We omit the proof here, but remark that due to compactness, the multiplicity $q$ of $\lambda_1$ is finite. In case $q>1$, the solution to \eqref{problem} is $\smash{ \sigma_1=\lambda_1^{1/2} }$ and the maximizing functions are in the eigenspaces corresponding to $\lambda_1$.
\end{remark}

In case that $T$ is a transport map, we have the following convergence behavior of $K_\varepsilon$ 
and $\nu_\varepsilon$ as $\varepsilon \rightarrow 0$, stating that the system associated with $K_{\varepsilon}$ is a \emph{small random perturbation} of $T$, see also~\citet{Kha63} and \citet{Kif86}.

\begin{proposition}\label{prop:conv}
Let $\mu$ and $\nu$ be absolutely continuous measures w.r.t.~the Lebesgue measure with densities
$\sigma_\mu$ and $\sigma_\nu$, respectively. 
Assume that the densities are positive a.e.~and that $\sigma_\mu  \in L_\infty(\X)$ is continuous a.e.
Let $T\colon \X \to \Y$ be a measurable map with $\nu = T_{\#}\mu$
and set $\pi \coloneqq (\id_\X,T)_{\#}\mu$.
Then, the kernel $K_{\varepsilon}$ and the measure $\nu_\varepsilon$ defined with respect to $T$ satisfy
$K_{\varepsilon}\, (\mu \otimes \nu_\varepsilon) \rightharpoonup \pi$ and $\nu_\varepsilon \rightharpoonup \nu$ as $\varepsilon \to 0$.
\end{proposition}

\begin{proof}
	For the first claim, it suffices to show 
	$\lim_{\varepsilon \to 0} K_{\varepsilon}\, (\mu \otimes \nu_\varepsilon)(A \times B) = \pi(A \times B)$ 
	for all $A \in \mathcal B(\X)$ and $B \in \mathcal B(\Y)$ with $\mu(\partial A) = 0$ and $\nu(\partial B)=0$, 
	see~\citet[Thm.~2.8(i)]{Billingsley99}.
	Due to our assumptions on the densities, these sets also satisfy $\lambda(\partial A)=0$ and $\lambda(\partial B)=0$.
	Using the definition of $\nu_\varepsilon$ and Fubini's theorem, we get
	\begin{align*}
		K_{\varepsilon} (\mu \otimes \nu_\varepsilon) (A \times B) 
		&= 
		\int_A \int_B k_\varepsilon (x,y) \dx y \dx \mu(x)
		= \int_A \int_B \int_\X \eta_{\varepsilon,x} (z) \eta_{\varepsilon,T(z)} (y) \dx z \dx y \dx \mu(x)\\
		&= \int_\X \int_A \eta_{\varepsilon,x}(z) \dx \mu(x) \, \int_B \eta_{\varepsilon, T(z)}(y)
		\dx y\,  \dx z.
	\end{align*}
	Next, we define
	\begin{align}
		f_{\varepsilon,A}(z) &\coloneqq \int_A \eta_{\varepsilon,x}(z) \dx \mu(x) , \qquad
		g_{\varepsilon,B}(z) \coloneqq \int_B \eta_{\varepsilon,T(z)}(y) \dx y.
	\end{align}
	First, we show $f_{\varepsilon,A}(z) \to \sigma_\mu(z)1_A(z)$ for a.e.~$z \in \X$.
	Fix $z \in \mathring{\X} \setminus \partial A$.
	Then, it holds for $\varepsilon$ small enough that  
	$\eta_{\varepsilon, x}(z) = 1_{B_\varepsilon(z)}(x)/\lambda(B_\varepsilon(z))$.
	Thus, we get for every such $z$ where $\sigma_\mu$ is continuous,
	\[
	\lim_{\varepsilon \to 0} f_{\varepsilon,A}(z) 
	= \lim_{\varepsilon \to 0} \int_A \frac{1_{B_\varepsilon(z)}(x)}{\lambda(B_\varepsilon(z))} \sigma_\mu(x) \dx x 
	= \lim_{\varepsilon \to 0} \frac{1_A(z)}{\lambda(B_\varepsilon(z))} \int_{B_\varepsilon(z)} \sigma_\mu(x)\dx x = \sigma_\mu(z)1_A(z)
	\]
	and hence convergence a.e.
	
	Next, we show $g_{\varepsilon,B}(z) \to 1_{T^{-1}(B)}(z)$ a.e. 
	For $z \in T^{-1}(\mathring B)$ and $\varepsilon$ small enough, we conclude that $g_{\varepsilon,B}(z) = 1$.
	If $z \in T^{-1}(\Y \setminus \text{cl}(B))$, we get for $\varepsilon$ small enough that $g_{\varepsilon,B}(z) = 0$.
	As $T$ is non-singular, the set $T^{-1}(\partial B)$ has Lebesgue measure zero and the claim follows. 
	
	By definition, we realize that $g_{\varepsilon,B}(z) \leq 1$ and by \eqref{eq:boundary} that
	$f_{\varepsilon,A}(z) \leq \Vert \sigma_\mu\Vert_{L_\infty(\X)} C_\X$.
	Now, the first assertion of the theorem follows by applying Lebesgue's dominated convergence theorem
	\begin{align*}
		\lim_{\varepsilon \to 0} K_{\varepsilon}\, (\mu \otimes \nu_\varepsilon) (A \times B)
		&= \lim_{\varepsilon \to 0} \int_\X f_{\varepsilon,A}(z) g_{\varepsilon,B}(z) \dx z = \int_\X 1_A(z) 1_{T^{-1}(B)}(z) \dx \mu(z) = \pi(A \times B).
 	\end{align*}
 	
 	In order to show $\nu_\varepsilon \rightharpoonup \nu$, it suffices to prove $\lim_{\varepsilon \to 0} \nu_\varepsilon(B) \to \nu(B)$ for all $B \in \mathcal B(\Y)$ with  $\nu(\partial B)=0$.
 	Repeating the same calculations as in the first part of the proof with $A = \X$, we obtain
 	\[\lim_{\varepsilon \to 0} \nu_\varepsilon(B)
 	= \lim_{\varepsilon \to 0} \int_B \int_\X k_\varepsilon(x,y) \dx \mu(x) \dx y 
	= \int_\X 1_{T^{-1}(B)}(z) \dx \mu(z) 
	\stackrel{\eqref{eq:push_f}}{=} \int_\Y 1_B (y) \dx \nu(y)
	= \nu(B).\]
\end{proof}

By Proposition \ref{prop:conv} and Theorem \ref{thm:tmap} we could use the OT map $\hat T$ induced by $\hat \pi$ and create a smoothed kernel with respect to this map. 

\subsection{Kernels from Regularized OT}\label{sec:frobenius_perron_3}
Having the OT results in mind, we propose to construct kernels based on regularized OT plans 
$\hat \pi_\varepsilon$ in \eqref{eq:PDrelation}, respectively  $\hat \pi_{\varepsilon, \kappa}$ in \eqref{eq:ot-unbalanced}
with the properties~\ref{enum:nu_marg}--\ref{enum:cpt_simple} and to use them as kernels
for transfer operators $L_\varepsilon$ in \eqref{FP_kernel}.
A motivation from statistical physics for this approach is given in Section~\ref{sec:schroedinger}.
Focusing on $\hat \pi_\varepsilon$, we use 
\begin{equation} \label{sm_kern_ot}
K_\varepsilon (x,y)
\coloneqq 
\frac{\dx \hat \pi_\varepsilon(x,y)}{\dx(\mu \otimes \nu)(x,y)} 
= 
\exp\Bigl(\frac{\hat \varphi_\varepsilon(x) + \hat \psi_\varepsilon(y) - c(x,y)}{\varepsilon}\Bigr) .
\end{equation}
Indeed, we can use this kernel for our segmentation problem \eqref{problem}, as the following proposition shows.
\begin{proposition}
For a Lipschitz cost function $c$, the kernel $K_\varepsilon$ and the associated operator $L_\varepsilon$ fulfill \ref{enum:nu_marg}--\ref{enum:cpt_simple}.
\end{proposition}
\begin{proof}
Using the marginal property of $\hat \pi_\varepsilon(x,y) = K_\varepsilon(x,y)  \sigma_\mu(x) \sigma_\nu(y) \lambda \otimes \lambda$, we obtain by integration that $\int_\X K_\varepsilon (x,y) \sigma_\mu(x) \sigma_\nu(y) \dx x = \sigma_\nu(y)$ and hence
\begin{align*}
\int_\X K_\varepsilon (x,y) \dx \mu(x) = \int_\X K_\varepsilon (x,y) \sigma_\mu(x) \dx x =  1 \quad \nu\text{-a.e.},
\end{align*}
which is \ref{enum:nu_marg} and similarly we show \ref{enum:mu_marg}.
Further, it is well known, that for Lipschitz $c$ the Kantorovich potentials are Lipschitz and hence bounded on compact domains, see \citet[Prop.~5.7]{neumayer2020optimal}. Therefore, square-integrability and non-negativity follow directly from \eqref{eq:PDrelation}.
Finally, the assumptions of Lemma $3$ in \citet{froyland2013analytic} are fulfilled as $K_\varepsilon > 0$ and
consequently the largest eigenvalue of $\mathcal L_\varepsilon = L_\varepsilon^*L_\varepsilon$ is simple, i.e, \ref{enum:cpt_simple} holds.
\end{proof}

A counterpart of Proposition \ref{prop:conv} is given below, see~\citet{CDPS17}.
Recall Theorem \ref{thm:tmap} for the uniqueness of the optimal transport plan 
and its relation to the optimal transport map.

\begin{proposition}\label{prop:conv-reg}
Let $\mu,\nu\in \mathcal P (\X)$, where $\mu$ is absolutely continuous 
with respect to the Lebesgue measure and let $c(x,y)=h(x-y)$ with a strictly convex function $h$.
Then, it holds for the kernel $K_{\varepsilon}$ in \eqref{sm_kern_ot} that
$\hat \pi_\varepsilon = K_{\varepsilon}\, (\mu \otimes \nu) \rightharpoonup \hat \pi$, 
where $\hat \pi$ is the minimizer of the original OT problem \eqref{Monge_Kantorovich_problem}.
\end{proposition}

\section{Motivation for Regularized OT Kernels from Statistical Physics} \label{sec:schroedinger}
In this section, we motivate that under certain assumptions our construction of $L_\varepsilon$ from the regularized OT plan is a reasonable choice.
For these purposes, the measures $\mu$ and $\nu$ are modeled as empirical measures corresponding to a large ensemble of $n$ indistinguishable particles with time-dependent random positions $(X^i_t)_t$, $i=1,\ldots,n$, $t\in \{0, 1\}$.
Then, any transition kernel choice for the transfer operator corresponds to a (joint) distribution of the particle ensemble.
In the absence of any other given dynamical information, it is unclear which joint distribution to pick.
Although this is unlikely to represent the true dynamics when conditioning on the given measures $\mu$ and $\nu$, one might therefore ask what happens if one makes the ``blind'' generic choice that each individual particle is simply following an independent Brownian path.
We show that this assumption interestingly leads precisely to the regularized optimal transport plan $\hat \pi_\varepsilon$ for the particle ensembles as a whole as the number of particles $n$ approaches infinity.
Although only the time-discrete case $t\in \{0, 1\}$ is considered, we note that the argumentation can be generalized to Wiener processes or even arbitrary positive path measures and state spaces.
Here, we refer to the survey of \citet{leonard2013survey} on Schr\"odinger's question, see \eqref{sch}, and to \citet{leonard2010entropic} for rigorous proofs.
Our self-contained approach may be more accessible than those for the general setting.

Let the particle positions be described by i.i.d.~$\R^d \times \R^d$-valued random vectors $(X_0^i,X_1^i)$, $i=1,\ldots,n$,
on a common probability space with conditional probability density
\begin{equation}\label{joint_density}
\sigma_{X_1 | X_0}^\varepsilon(x_0, x_1) 
 = \tfrac{1}{\sqrt{\pi \varepsilon}^d}\exp \left(-\Vert x_0 - x_1\Vert^2 / \varepsilon \right),
\quad \varepsilon > 0.
\end{equation}
Since we assume that the particles are indistinguishable, 
the ensemble of particles at time $t \in \{0,1\}$ can be described using a $\mathcal P(\R^d)$-valued random variable, 
in other words, a random probability measure,
\begin{equation} \label{Lnt}
 Z_t^n \coloneqq \frac{1}{n} \sum_{i=1}^{n} \delta_{X_t^i}.
\end{equation}
Note that $\mathcal P(\R^d)$ shall again be equipped with the weak topology and we rely on the corresponding Borel $\sigma$-algebra.
The following proposition gives an intuition on the convergence behavior of $Z^n_1$ as $n\rightarrow \infty$
for such a process conditioned on starting points $X_0^i = x_0^i$ for which the corresponding empirical distribution converges weakly to $\mu$. 
Note that we re-use $Z^n_0$ with a slight abuse of notation.

\begin{proposition} \label{rem:nice}
For any sequence of sampled initial points $(x_0^i)_{i \in \N}$ with
\begin{equation} \label{eq:weak_1} 
Z_0^n = \frac{1}{n} \sum_{i=1}^{n} \delta_{x_0^i} \rightharpoonup \mu,
\end{equation}
the empirical random probability measure
\[\tilde Z^n_1 \coloneqq \frac1n \sum_{i=1}^n \delta_{X_1^i | X_0^i = x_0^i}\]
converges a.s.\ to the constant random measure $(\sigma^\varepsilon_{X_1 | X_0}(\cdot, 0) * \mu) \lambda$ as $n\to \infty$.
In other words,  for a.e.~realization of
the random variables $(X_1^i| X_0^i = x_0^i)_{i\in \N}$
the associated empirical measure converges weakly to $(\sigma^\varepsilon_{X_1 | X_0}(\cdot, 0) * \mu) \lambda$ as $n\to \infty$.
\end{proposition}
\begin{proof}
For any $f \in C_b(\mathbb R^d)$, 
the random variables $f(X_1^i| X_0^i = x_0^i)$, $i \in \mathbb N$, are independent with finite variances $V^i$ fulfilling 
$\sum_i V^i/i^2 < \infty$, so that we obtain by Kolmogorov's strong law of large numbers \citep[p.~389]{Shiryaev1996} that
\begin{align} 
\int_{\R^d} f(y) \dx \tilde Z^n_1(y) = &\frac{1}{n}\sum_{i=1}^n f(X_1^i|X_0^i = x_0^i) 
\xrightarrow{\text{a.s.}} 
\lim_{n \to \infty} \frac{1}{n}\sum_{i=1}^n E\bigr(f (X_1^i|X_0^i = x_0^i)\bigr).
\label{auch_noch}
\end{align}
By \eqref{joint_density}, we have
$$
\bigl(\sigma_{X_1 | X_0}^\varepsilon (\cdot,0) * Z_0^n\bigr)(y) 
= \frac{1}{n}\sum_{i=1}^n \sigma_{X_1 | X_0}^\varepsilon (y-x_0^i, 0)
= \frac{1}{n}\sum_{i=1}^n \sigma_{X_1 | X_0}^\varepsilon (x_0^i, y) ,
$$
so that by \eqref{eq:weak_1} and weak continuity of convolutions and $Z_0^n\rightharpoonup \mu$, the limit in \eqref{auch_noch} becomes
\begin{align}
&\lim_{n \to \infty} \frac{1}{n}\sum_{i=1}^n E\bigr(f (X_1^i|X_0^i = x_0^i)\bigr)
=
\lim_{n \to \infty} \int_{\R^d} f(y) \frac{1}{n}\sum_{i=1}^n \sigma_{X_1 | X_0}^\varepsilon (x^i_0, y) \dx y\\
&= \lim_{n \to \infty} \int_{\R^d} f(y) \bigl(\sigma_{X_1 | X_0}^\varepsilon (\cdot,0) * Z_0^n\bigr)(y) \dx y 
= \int_{\R^d} f(y) \bigl(\sigma_{X_1 | X_0}^\varepsilon(\cdot, 0) * \mu\bigr)(y) \dx y.
\end{align}
Thus, assuming \eqref{eq:weak_1}, we obtain the assertion.
\end{proof}

In contrast to the setting in the proposition, we  assume now that the initial configuration at $t=0$ lies 
in a small Wasserstein-2 ball 
$\Bmuxi \coloneqq \{\alpha\in \PRd : W_2(\alpha, \mu) \leq r \}$ with radius $r>0$ 
around $\mu$, and those of the observed particle distribution at $t=1$ 
in $\Bnuxi$.
The collective dynamical behavior of the ensemble is described by 
the $\mathcal{P}(\R^d \times \R^d)$-valued random variable
\begin{equation} \label{Ln} 
Z^n \coloneqq \frac{1}{n} \sum_{i=1}^{n} \delta_{(X_0^i, X_1^i)},
\end{equation}
which contains more information than the two random probability measures $Z^n_t$, $t=0,1$.
We are interested in estimating the conditional probability
\begin{align}\label{eq:cond_prob}
P\bigl(Z^n \in A \,\vert\, Z_0^n\in \Bmuxi, Z_1^n\in \Bnuxi\bigr) 
= \frac{P\bigl(Z^n \in A,  Z_0^n\in \Bmuxi, Z_1^n\in \Bnuxi\bigr)}{P\bigl(Z_0^n\in \Bmuxi, Z_1^n\in \Bnuxi\bigr)}
\end{align}
for large $n$ and all $A \in \mathcal{B}(\mathcal{P}(\R^d \times \R^d))$. For small $r$ and in the infinite particle limit, this is precisely
\emph{Schr\"odinger's question}, that can now be stated as follows: What is the limit
\begin{equation}\label{sch}
\lim_{r\to 0}\lim_{n\to \infty} P\bigl(Z^n \in A \,\vert\, Z_0^n\in \Bmuxi, Z_1^n\in \Bnuxi\bigr)?
\end{equation}
Interestingly, the limiting dynamical behavior of the particles is described approximately by the regularized optimal transport plan $\hat\pi_\varepsilon$; more precisely, the limit above and answer to Schr\"odinger's question is $\delta_{\hat\pi_\varepsilon}(A)$.
However, we have to be careful here. 
The statement holds for the \emph{entropy regularized OT distance}
\begin{align}\label{eq:ot_ent}
    \OT_{\varepsilon}(\mu,\nu) 
&\coloneqq  
\min_{\pi \in \Pi(\mu,\nu)} 
\,
\int_{\X \times \Y}  c\dx \pi
+ \varepsilon \mathrm{KL} (\pi,\lambda \otimes \lambda)\\
&= 
\varepsilon
\min_{\pi \in \Pi(\mu,\nu)} 
\, \mathrm{KL} \bigl(\pi, \exp(-c/\varepsilon) \lambda \otimes \lambda\bigr) - \varepsilon \int_{\X \times \Y}   \exp(-c/\varepsilon) -1 \dx  (\lambda \otimes \lambda).
\end{align}
The relation between the KL-regularized OT in~\eqref{sinkhorn} and the
entropy regularized OT in~\eqref{eq:ot_ent} is described in the following
remark. Roughly speaking, the KL-regularized OT is more general as minimizers always exist, 
but there are many cases where both minimization problems coincide.

\begin{remark}\label{rem:entropy}
Let $\sigma_\mu$ and $\sigma_\nu$ denote the densities of $\mu$ and $\nu$, respectively.
For $\pi \ll \lambda \otimes \lambda$ with density $\sigma_\pi$ the entropy is defined by
$E(\pi) \coloneqq \mathrm{KL}(\pi,\lambda \otimes \lambda)$.
Note that 
$ \pi \ll \mu \otimes \nu$ if and only if $\pi \ll \lambda \otimes \lambda$ for any $\pi \in \Pi(\mu,\nu)$.
If $\mathrm{KL} (\mu \otimes \nu, \lambda \otimes \lambda) < \infty$, we can show for any $\pi \ll \lambda \otimes \lambda$ 
with $\pi \in \Pi(\mu,\nu)$ that it holds
\begin{align}
	\mathrm{KL} (\pi,\lambda \otimes \lambda) - \mathrm{KL} (\mu \otimes \nu, \lambda \otimes \lambda) \label{care} =\mathrm{KL} (\pi,\mu \otimes \nu).
\end{align}
Consequently, in this case the minimizer for KL-regularized OT~\eqref{sinkhorn} and entropy regularized OT~\eqref{eq:ot_ent} coincide.
The crux is the condition $\mathrm{KL} (\mu \otimes \nu, \lambda \otimes \lambda) < \infty$, which is equivalent to $\mu,\nu$ having finite entropy,
i.e.,  $\sigma_\mu, \sigma_\nu$ are in a so-called Orlicz space $L \log L$, see~\citet{CLMW19,NR2013}.
For more information we refer to \citet{neumayer2020optimal}.
\end{remark}

The following proposition establishes a relation between regularized optimal transport plans minimizing \eqref{eq:ot_ent} and Schr\"odinger's question.
As already mentioned, this fact is known in a more general setting.
For the sake of completeness, we add the proof for our setting.

\begin{proposition}\label{prop:schroedinger}
Let $(X_0^i,X_1^i)_{i\in \N}$ be i.i.d.~$\R^d \times \R^d$-valued random vectors
on a common probability space with conditional density \eqref{joint_density}, and let $Z^n_t$, $t = 0,1$, and $Z^n$ be given by \eqref{Lnt} and \eqref{Ln}, respectively.
Then, it holds for $\mu,\nu \in \mathcal P(\mathbb R^d)$ fulfilling $\mathrm{KL} (\mu \otimes \nu, \lambda \otimes \lambda) < \infty$ and every $A \in \mathcal B(\mathcal P(\R^d \times \R^d))$ with $\hat \pi_\varepsilon \notin \partial A$ that
\begin{equation}\label{schroedinger}
\lim_{r\to 0} \lim_{n\to \infty} P\bigl(Z^n \in A \,\vert\, Z_0^n \in \Bmuxi, Z_1^n \in \Bnuxi \bigr) 
= 
\delta_{\hat \pi_\varepsilon}(A),
\end{equation}
where $\hat \pi_\varepsilon$ is the regularized optimal transport plan minimizing \eqref{eq:ot_ent} with the Wasserstein-2 cost function.
\end{proposition}
\begin{proof}
For $r>0$, 
set $C_r \coloneqq \{\pi \in \mathcal P(\R^d\times \R^d):\pi_0 \in \Bmuxi, \pi_1 \in \Bnuxi\}$ 
and consider the unique minimizers $\hat \pi_{r,\varepsilon} \coloneqq 
\argmin_{\pi \in C_r} \mathrm{KL}(\pi ,\sigma^\varepsilon_{X_1 | X_0} \lambda \otimes \lambda)$, see after \eqref{KLdef}.
Clearly, we have by \eqref{sinkhorn_kernel} for any $r>0$ that
\begin{align} \label{soso}
\mathrm{KL}(\hat \pi_{r,\varepsilon} ,\sigma^\varepsilon_{X_1 | X_0}\lambda \otimes \lambda) 
= \min_{\pi \in C_r} \mathrm{KL}(\pi ,\sigma^\varepsilon_{X_1 | X_0}\lambda \otimes \lambda) \le 
\mathrm{KL}(\hat \pi_\varepsilon ,\sigma^\varepsilon_{X_1 | X_0}\lambda \otimes \lambda).
\end{align}
Choose a sequence $(r_j)_{j\in \N}$, $r_j>0$, with $r_j\to 0$ for $j\to \infty$.
Using the weak compactness of $\mathcal P(\R^d\times \R^d)$ and the weak closedness of any $C_r$,
we conclude that all accumulation points of 
$(\hat\pi_{r_j,\varepsilon})_j$ are contained in $C_{r_j}$ for any fixed $j\in \N$.
Consequently, they are also contained in $\Pi(\mu,\nu) = \cap_{j\in\N} C_{r_j}$.
Take a weak accumulation point $\tilde\pi_\varepsilon$ using weak compactness,
and with abuse of notation, choose a subsequence $(r_j)_{j\in \N}$ such that $\hat\pi_{r_j,\varepsilon} \rightharpoonup \tilde \pi_\varepsilon \in \Pi(\mu,\nu)$ as $r_j\to 0$.
By weak lower-semicontinuity of $\mathrm{KL}(\cdot,\sigma_{X_1 | X_0}^\varepsilon \lambda \otimes \lambda)$, 
we have
that
\[
\mathrm{KL}\bigl(\tilde \pi_\varepsilon ,\sigma^\varepsilon_{X_1 | X_0}\lambda \otimes \lambda \bigr) 
\le 
\lim_{j\to \infty}  \mathrm{KL}\bigl(\hat \pi_{r_j,\varepsilon} ,\sigma^\varepsilon_{X_1 | X_0}\lambda \otimes \lambda\bigr) 
\le \mathrm{KL} (\hat \pi_\varepsilon,\sigma^\varepsilon_{X_1 | X_0}\lambda \otimes \lambda \bigr).
\]
Since $\tilde \pi_\varepsilon$ is a feasible point of the regularized OT problem \eqref{sinkhorn},
this implies that the whole sequence
$(\hat \pi_{r_j,\varepsilon})_j$ converges weakly to $\tilde \pi_\varepsilon =  \hat \pi_\varepsilon$ as $j\to\infty$.
Due to \eqref{eq:cond_prob}, we have
\begin{align}
\lim_{n\to \infty} \frac{1}{n} \log P\bigl(Z^n \in A \,\vert\, Z_0^n\in \Bmuxi, Z_1^n\in \Bnuxi\bigr) = \lim_{n\to \infty} \frac{1}{n}\log \frac{P(Z^n \in A \cap C_r)}{P(Z^n\in C_r)}.
\end{align}
Using logarithm laws and Sanov's Theorem~\ref{thm:sanov}, see~\citet[Thm.~6.2.10]{DZ2010}, for the respective summands, we obtain for any measurable $A$ satisfying $A = \text{cl} (\mathring A)$ 
and $\hat \pi_\varepsilon \notin \text{cl}(A)$--which by convergence of $(\hat \pi_{r_j,\varepsilon})_j$ also implies $\hat \pi_{r,\varepsilon} \notin \text{cl}(A)$ for $r>0$ small enough--that
\begin{align*}
\lim_{n\to \infty} \frac{1}{n}\log \frac{P(Z^n \in A \cap C_r)}{P(Z^n\in C_r)}
= \inf_{\pi \in C_r} \mathrm{KL}\bigl(\pi ,\sigma_{X_1 | X_0}^\varepsilon\lambda \otimes \lambda\bigr) -\inf_{\pi \in A \cap C_r}\mathrm{KL}\bigl(\pi ,\sigma_{X_1 | X_0}^\varepsilon\lambda \otimes \lambda\bigr) <0,
\end{align*}
where the last inequality follows from the strict convexity of $\mathrm{KL}$.
Hence, we obtain
\[\lim_{r\to 0}\lim_{n\to \infty} P\bigl(Z^n \in A \,\vert\, Z_0^n\in \Bmuxi, Z_1^n\in \Bnuxi\bigr) = 0.\]
Using complements, we conclude for any measurable $A$ satisfying $A = \text{cl} (\mathring A)$ 
and $\hat \pi_\varepsilon \in \mathring A$, 
that
\[
\lim_{r\to 0}\lim_{n\to \infty} P\bigl(Z^n \in A \,\vert\, Z_0^n\in \Bmuxi, Z_1^n\in \Bnuxi\bigr) = 1,
\]
so that our claim follows from these two results using the monotonicity of measures.
\end{proof}

\begin{theorem}[Sanov]\label{thm:sanov}
Let $(X_i)_{i \in \N}$ be a sequence of i.i.d.~$\mathcal P(\X)$-valued random variables 
with distribution $m$ on some probability space, where $\mathcal P(\X)$ is equipped with the weak topology.
Then, it holds for any measurable $A \in \mathcal B(\mathcal P(\X))$ with $A = \text{cl}(\mathring A)$ that
\[\lim_{n \to \infty} \frac1n \log P\biggl( \frac1n\sum_{i=1}^n \delta_{X_i} \in A \biggr) = - \inf_{\pi \in A} \mathrm{KL}(\pi, m).\]
\end{theorem}
To summarize, in the case that the particles make independent Gaussian jumps,
conditioning on initial and end configurations close to $\mu$ and $\nu$ means 
having a dynamical behavior of the particles described approximately by the regularized optimal transport plan~$\hat \pi_\varepsilon$.

\section{Discretization}\label{sec:discrete}
In this section, we discuss the discrete settings for our numerical computations.
In Subsection~\ref{subsec:STK}, we consider the construction of smoothed transition kernels from OT plans and from entropic OT plans and illustrate their behaviour by a numerical example.
Since the second approach is much more efficient, we will choose it in the  numerical part.
Further, we provide the corresponding segmentation algorithms in  Subsection~\ref{subsec:entropic}.

\subsection{Discrete Kernels}\label{subsec:STK}
In view of numerical applications, we focus now on discrete OT for the kernel construction.
Let $(x_i)_{i \in I}$ with $I = \{i=1,\ldots,m\}\subset \Z$ and $(y_j)_{j \in J}$ with $J = \{j=1,\ldots,n\}\subset \Z$ be the support points of 
$\mu = \sum_{i \in I} \mu_i \delta_{x_i}$ and $\nu = \sum_{j \in J} \nu_j \delta_{y_j}$, respectively.
Then, the OT plan is given by
$$
\hat \pi = \argmin_{\pi \in \Pi(\mu,\nu)} \sum_{i \in I} \sum_{j \in J} c(x_i,y_j) \pi(i,j),
$$
and similarly for the regularized OT plans.
Here, $\hat \pi$ can be interpreted as a mapping $\hat \pi\colon I \times J \to \R_{\ge 0}$. 
For convenience, 
we extend all vectors and mappings to the whole integers by setting them to zero outside of the index sets $I$ and $J$.

\paragraph{Smoothed Transition Kernels from OT}
The construction of smoothed transition kernels as proposed in Subsection~\ref{sec:frobenius_perron_2} 
relies on the transport map $\hat T$ associated to the OT plan~$\hat \pi$. 
Unfortunately, for the discrete transport problem,
this OT map does not necessarily exist.
Therefore, we replace the smoothed transition kernels described in Subsection~\ref{sec:frobenius_perron_2} 
by a construction that uses transport plans instead of transport maps.
Let $\sigma_{\varepsilon}\colon \Z \to \R$ denote some positive, normalized smoothing kernel 
centered around $i=0$ with finite width. For $i \in I$, we smooth $\hat \pi$ in $j$-direction to get
\[
K^1_{\hat \pi,\varepsilon}(i,j) \coloneqq \biggl(\frac{\hat \pi(i,\cdot)}{\mu_i\sum_{k \in J} \sigma_\varepsilon(k-\cdot)} \ast \sigma_\varepsilon\biggr) (j), \quad j \in J 
\]
and set $K^1_{\hat\pi,\varepsilon}(i,j) \coloneqq 0$ for $j \not \in J$ or $i \not \in I$.
The rescaling of $\hat \pi$ ensures that mass is preserved, i.e., 
\begin{align}
\sum_{j \in J} K^1_{\hat \pi,\varepsilon}(i,j) 
&= \sum_{j \in J} \sum_{r \in \mathbb Z} \frac{\hat \pi(i,r)}{\mu_i\sum_{k \in J} \sigma_\varepsilon(k-r)}  \sigma_\varepsilon(j-r) = \sum_{r \in \mathbb Z} \frac{\hat \pi(i,r)}{\mu_i\sum_{k \in J} \sigma_\varepsilon(k-r)} \sum_{j \in J}\sigma_\varepsilon(j-r) = 1.
\end{align}
Next, we smooth in $i$-direction as
\[
K^2_{\hat \pi,\varepsilon}(i,j) \coloneqq \frac{K^1_{\hat \pi,\varepsilon}(\cdot,j) \ast \sigma_\varepsilon (i)}{\sum_{k \in I}\sigma_\varepsilon(i-k)}
\]
and set again $K^2_{\hat \pi,\varepsilon}(i,j)\coloneqq 0$ for $j \not \in J$ or $i \not \in I$. Here, the denominator ensures that the mass is only distributed between indices $i\in I$ 
and consequently, for any $i \in I$,
\begin{align*}
\sum_{j \in J} K^2_{\hat\pi,\varepsilon}(i,j) 
=  
\sum_{j \in J} \frac{\sum_{r \in \mathbb Z} K^1_{\hat \pi,\varepsilon}(r,j) \sigma_\varepsilon(i-r)}{\sum_{k \in I}\sigma_\varepsilon(i-k)} = \frac{\sum_{r \in I} \sigma_\varepsilon(i-r) }{\sum_{k \in I}\sigma_\varepsilon(i-k)} \sum_{j \in J}  K^1_{\hat \pi,\varepsilon}(r,j)
= 1.
\end{align*}
Then, the final kernel is defined by
\[K_{\hat \pi,\varepsilon}(i,j) \coloneqq \frac{K^2_{\hat \pi,\varepsilon}(i,j)}{\sum_{i \in I} K^2_{\hat \pi,\varepsilon}(i,j) \mu_i}.\]
It is straightforward to check that this kernel fulfills \ref{enum:nu_marg}, 
and also \ref{enum:mu_marg}, with the smoothed marginal measure $\nu^\varepsilon_j = \sum_{i \in I} K^2_{\hat\pi,\varepsilon}(i,j) \mu_i$, i.e., $\sum_{i \in I} K_{\hat\pi,\varepsilon}(i,j) \mu_i = 1$ and $\sum_{j \in J} K_{\hat\pi,\varepsilon}(i,j) \nu^\varepsilon_j = 1$.
Property~\ref{enum:cpt_simple}, that is, simplicity of the largest singular value of $L_\varepsilon$ defined by \eqref{FP_kernel}, might fail in cases where $\hat\pi$ has a ``block-diagonal structure'' and the width of $\sigma_\varepsilon$ is small enough for the blurred kernel $K_{\hat \pi,\varepsilon}$ to retain this structure. In this case, either $\varepsilon$ needs to be increased, or one obtains ``perfectly'' coherent sets corresponding to the largest singular value, see Remark~\ref{rem:simplicity}.

\paragraph{Kernels from Regularized OT}
Having computed the OT plan $\hat \pi_\varepsilon > 0$  of the discrete regularized OT,
the corresponding kernel for our transfer operator is given by 
\begin{equation} \label{sm_kern_ot_d}
K_\varepsilon (i,j)
\coloneqq 
\frac{\hat \pi_\varepsilon(i,j)}{\mu(i)\nu(j)} , \quad i \in I, j \in J.
\end{equation}
For plans $\hat \pi_{\varepsilon,\kappa}$ arising from
unbalanced OT we have to choose
\begin{equation} \label{unbal_kern_ot_d}
K_\varepsilon (i,j)
\coloneqq 
\frac{\hat \pi_{\varepsilon,\kappa} (i,j)}{\tilde \mu(i) \tilde \nu(j)} , \quad i \in I, j \in J,
\end{equation}
where 
$\tilde \mu = \sum_j \hat \pi_{\varepsilon,\kappa} (\cdot,j)$ 
and 
$\tilde \nu = \sum_i \hat \pi_{\varepsilon,\kappa} (i,\cdot)$
are the marginals of $\hat \pi_{\varepsilon,\kappa}$.
Then, the properties \ref{enum:nu_marg}--\ref{enum:cpt_simple} are again ensured with respect to $\tilde \mu$ and $\tilde \nu$.

\paragraph{Numerical Comparison}
The proposed methods for creating kernels are compared numerically for the cost $c(x_i,y_j) = (x_i-y_j)^2$.
\renewcommand\curfolder{img/kernel_comparison}
\begin{figure}[p!]
	\centering
	\includegraphics[width=0.9\textwidth]{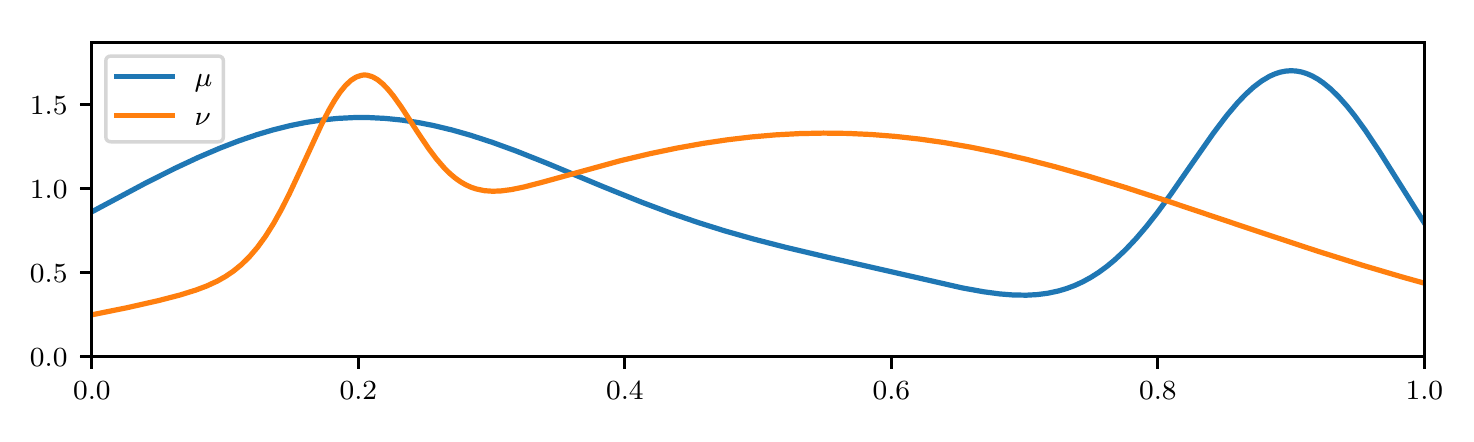}
	\caption{Probability distributions $\mu$ and $\nu$ for comparing different transition kernel constructions.}
	\label{fig:kernel_comparison_data}
	\vspace{.5cm}
	\centering
	\includegraphics[clip, trim=1.5cm 1.5cm .8cm 1.5cm, width=0.9\textwidth]{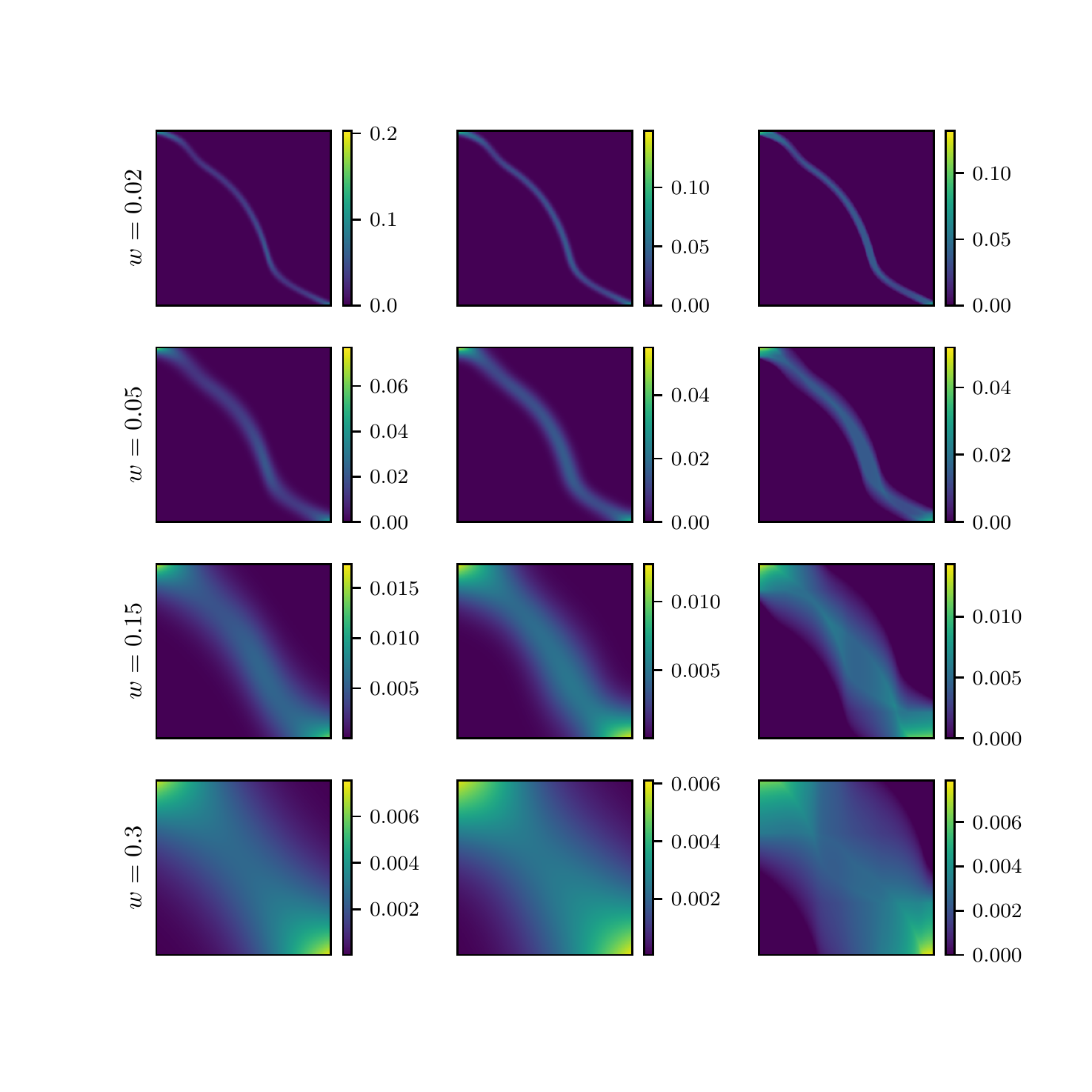}
	\caption{Comparison of kernels from regularized OT (left) and smoothed kernels from OT with a Gaussian blur (middle) and a ball-averaging one (right). The vertical direction belongs to the first and the horizontal direction to the second component of the kernels under comparison.
	The blur width $w = \sqrt{\varepsilon / 2}$ is adapted to the regularization parameter $\varepsilon$ of the regularized OT, such that all kernels have equal bandwidth.
	}
	\label{fig:kernel_comparison}
\end{figure}
For this purpose, we fix the two probability densities $\mu$ and $\nu$ on $[0,1]$ 
displayed in Fig.~\ref{fig:kernel_comparison_data}, 
which are sums of Gaussians.
The kernel proposed in \eqref{sm_kern_ot_d} for different standard deviations $w = \sqrt{\varepsilon/2}$ is shown in the left column of Fig.~\ref{fig:kernel_comparison}.
To get a clue about the smoothed kernel, we blur the OT plan $\hat \pi$ using two different  functions, namely a Gaussian kernel $\sigma_w^1\colon \Z \to \R$ 
and an averaging blur $\sigma_w^2\colon \Z \to \R$ within a ball
of width $\lfloor w \rfloor$, both centered around $0$.
To ensure that $\sigma_w^1$ has finite width, we set the values below $10^{-4}$ to zero.
Both discrete functions are normalized, such that they sum up to one.
The resulting smoothed kernels are depicted in the middle and right columns in Fig.~\ref{fig:kernel_comparison}.
They look similar to the left ones, however, the averaging kernel appears to be artificially rough for large $w$.
As a consequence, it seems natural to use regularized OT plans instead of smoothed OT plans.
Note that the cost for computing $\hat \pi_\varepsilon$ with the Sinkhorn algorithm scales with $1/\varepsilon$, 
i.e., $\varepsilon$ should not be to small, see~\citet{C2013}. This is not a real issue, since we need a certain amount of blur anyways in order to ensure that the leading singular value 
in the corresponding transfer operator $L_\varepsilon$ is simple.

\subsection{Segmentation Algorithms}\label{subsec:entropic} 
In our numerical examples in the next section, we apply the kernel \eqref{sm_kern_ot_d}. However, we will see that for the addressed applications
the unbalanced OT plan $\hat \pi_{\varepsilon,\kappa}$ with corresponding kernel \eqref{unbal_kern_ot_d} leads to more natural
results.
In the following, we use the matrix-vector notation
$$
\mu \coloneqq \left( \mu(i) \right)_{i=1}^m, \; \Sigma_\mu \coloneqq  \mathrm{diag} \, \mu,\;
\nu \coloneqq \left( \nu(j) \right)_{j=1}^n, \Sigma_\nu \coloneqq  \mathrm{diag} \, \nu,
$$
and $\|f\|_\mu \coloneqq \|\Sigma_\mu^{\scriptscriptstyle 1/2} f\|_2$, $\|g\|_\nu \coloneqq \|\Sigma_\nu^{\scriptscriptstyle 1/2} g\|_2$.
Further, let
$K_\varepsilon \coloneqq \left( K_\varepsilon(i,j) \right)_{i,j=1}^{m,n}$, 
$\hat \pi_\varepsilon \coloneqq \left( \hat \pi_\varepsilon(i,j) \right)_{i,j=1}^{m,n}$,
and similarly for the unbalanced kernels.
Then \eqref{sm_kern_ot_d} becomes 
$K_\varepsilon = \Sigma_\mu^{\scriptscriptstyle -1} \hat \pi_\varepsilon \Sigma_\nu^{\scriptscriptstyle -1}$.
To solve the segmentation model~\eqref{problem}, 
we have to find the second largest singular values of  the discrete transfer operator
$L_\varepsilon \colon (\mathbb R^m, \|\cdot\|_\mu) \to (\mathbb R^n, \|\cdot\|_\nu)$
given by
\[
L_\varepsilon f =  K_\varepsilon^\tT \Sigma_\mu f = \Sigma_\nu^{-1} \hat \pi_\varepsilon^\tT f.
\]
Then, the Rayleigh quotient in \eqref{problem} can be rewritten as
\begin{align}\label{svd}
\frac{ \langle L_\varepsilon f, g \rangle_\nu}{ \Vert f\Vert_{\mu} \Vert g \Vert_{\nu}}
&= 
\frac{ f^\tT  \Sigma_\mu K_\varepsilon \Sigma_\nu g }{\|\Sigma_\mu^{\scriptscriptstyle 1/2} f\|_2 \, \|\Sigma_\nu^{\scriptscriptstyle 1/2} g\|_2}
=
\frac{ f^\tT  \hat \pi_\varepsilon g }{\|\Sigma_\mu^{\scriptscriptstyle 1/2} f\|_2 \, \|\Sigma_\nu^{\scriptscriptstyle 1/2} g\|_2}
=
\frac{ u^\tT \sighalf_\mu \hat \pi_\varepsilon \sighalf_\nu v }{\|u\|_2 \, \|v\|_2},
\end{align}
where we substituted $u \coloneqq \Sigma_\mu^{\scriptscriptstyle 1/2} f$ and $v \coloneqq \Sigma_\nu^{\scriptscriptstyle 1/2} g$.
The constraint in \eqref{problem}  becomes 
$$0 = \langle f, 1_\X\rangle_\mu = f^\tT \Sigma_\mu 1_\X = u^\tT \Sigma_\mu^{\scriptscriptstyle 1/2} 1_\X = \langle u, \Sigma_\mu^{\scriptscriptstyle 1/2} 1_\X \rangle$$
and indeed $\Sigma_\mu^{\scriptscriptstyle 1/2} 1_\X$ is the left singular vector to the largest singular value of $\sighalf_\mu \hat \pi_\varepsilon \sighalf_\nu$.
This  holds also accordingly for the dominant right singular vector.
To compute the singular vectors belonging to the largest singular values of this matrix, we apply a (truncated) singular value decomposition (SVD).
We summarize our steps in Algorithm~\ref{alg:r-ot}.
Readers familiar with \citet{FrSaMo10} will notice the evident analogies between our algorithm and \citet[Lem.~1]{FrSaMo10} in computing the segmentation (coherent sets).
The transition matrix $P$ therein connects to the objects used by us via~$\Sigma_{\nu}^{-1}P^{\mathrm{T}} = K_{\varepsilon}^{\mathrm{T}}$.
If the marginal requirement on $\mu$ and $\nu$ is not fulfilled, unbalanced OT may be preferable.
Similarly as above, the same considerations for unbalanced regularized OT lead to  Algorithm~\ref{alg:u-r-ot}.
\begin{algorithm}[!tb]
\begin{algorithmic}
\State Input: Particle distributions $\mu$ and $\nu$ with $\mu^\tT 1_m = \nu^\tT 1_n$,
              $\varepsilon >0$.

\State 1. Compute the regularized OT plan $\hat{\pi}_\varepsilon$ using Sinkhorn's algorithm.

\State 2. Compute the first non-trivial singular vectors $u_2, v_2$ of  $\sighalf_\mu \pieps \sighalf_\nu$ by truncated SVD.

\State 3. Set \smash{$\hat{f} \coloneqq \sighalf_\mu u_2$} and \smash{$\hat{g} \coloneqq  \sighalf_\nu v_2$}.
\caption{Segmentation based on regularized OT}
\label{alg:r-ot}
\end{algorithmic}
\end{algorithm}

\begin{algorithm}[!tb]
\caption{Segmentation based on unbalanced regularized OT}
\label{alg:u-r-ot}
\begin{algorithmic}
\State Input: Particle distributions $\mu$ and $\nu$, $\varepsilon >0$, $\kappa > 0$.

\State 1. Compute the unbalanced regularized OT plan $\hat{\pi}_{\varepsilon, \kappa}$ using Sinkhorn's algorithm.\\
          \hspace*{0.3cm} Set $\tilde \mu \coloneqq \hat{\pi}_{\varepsilon, \kappa} 1_n$, $\tilde \nu \coloneqq \hat{\pi}_{\varepsilon, \kappa}^\tT 1_m$.

\State 2. Compute the first non-trivial singular vectors $u_2, v_2$ of  
\smash{$\sighalf_{\tilde \mu} \hat{\pi}_{\varepsilon, \kappa} \sighalf_{\tilde \nu}$} by truncated SVD.

\State 3. Set \smash{$\hat{f} \coloneqq \sighalf_{\tilde \mu} u_2$} and \smash{$\hat{g} \coloneqq \sighalf_{\tilde \nu} v_2$}.
\end{algorithmic}
\end{algorithm}

\begin{remark}[Multiphase Segmentation] \label{rem:multiphase}
In our numerical examples, we are also interested in partitions with more than
just two sets.
These can be obtained by using additional singular functions as it is also known in the extensive literature on graph cut partitions, see, e.g., \citet{Luxburg2007}.
To this end, we employ the maximizers of the problem
\begin{align} \label{third_eigenvector}
 \max_{
\substack{
(f,g) \in L_2(\X,\mu) \times L_2(\Y,\nu)
} }
\biggl\{ 
\frac{ \langle L f, g \rangle_\nu}{ \Vert f\Vert_{L_2(\X,\mu)} \Vert g \Vert_{L_2(\Y,\nu)} }: 
\langle f, 1_\X \rangle_\mu = \langle f, f_2 \rangle_\mu = \langle g, 1_\Y \rangle_\nu = \langle g, g_2\rangle_\nu = 0 
\biggr\},
\end{align}
where $f_2, g_2$ are the solutions of \eqref{problem}, i.e., the second largest singular vector pair.
It is not hard to show that the solutions of \eqref{third_eigenvector} are given by the singular functions corresponding to the third largest singular value of $L$. Similarly, we can consider
further singular function pairs.
Based on these singular functions, we compute a multiphase partition using the $c$-means (or fuzzy $c$-means) algorithm.
\end{remark}
\section{Numerical Results} \label{sec:numerics}
In this section, we present various examples.
All involved transport plans are computed with respect to the squared Euclidean distance as cost function.
\subsection{Exploring Coherent Set Detection}\label{subsec:precipitation}
\renewcommand\curfolder{img/data_1}
\renewcommand\curwidth{0.235\textwidth}
\begin{figure}[tp]
	\centering
	\subfloat[15 Apr 2018, 08:35am.\label{subfig:rain_1_t0}]{\includegraphics[width=\curwidth]{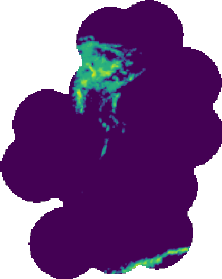}}
	\,
	\subfloat[15 Apr 2018, 08:55am.\label{subfig:rain_1_t1}]{\includegraphics[width=\curwidth]{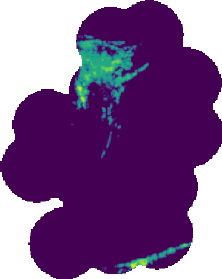}}
	\renewcommand\curfolder{img/data_2}
	\,
	\subfloat[16 Apr 2018, 06:00am\label{subfig:rain_2_t0}]{\includegraphics[width=\curwidth]{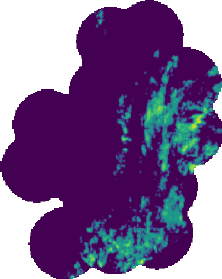}}
	\,
	\subfloat[16 Apr 2018, 06:30am.\label{subfig:rain_2_t1}]{\includegraphics[width=\curwidth]{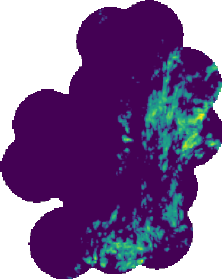}}
	\caption{Two precipitation density pairs over Germany. Plotted in logarithmic scale for better visibility of low-precipitation-areas.}
	\label{fig:data_clouds}
\end{figure}
First, we apply our method to a data set of precipitation densities over Germany, which was made freely available by \citet{dwd}.
We take two pairs of snapshots, see Fig.~\ref{fig:data_clouds}.
In order to cope with memory limitations, we have applied spatial averaging and masked out all indices $i, j$ with $\mu_i = 0$ or $\nu_j= 0$, resulting in the problem dimensions $\vert I\vert = 3930$ for Fig.~\ref{subfig:rain_1_t0}, $\vert J\vert = 3939$ for Fig.~\ref{subfig:rain_1_t1}, $\vert I\vert = 10725$ for Fig.~\ref{subfig:rain_2_t0} and $\vert J\vert = 10028$ for Fig.~\ref{subfig:rain_2_t1}.
In Figs.~\ref{subfig:rain_1_t0}--\ref{subfig:rain_1_t1}, we clearly see two main areas of precipitation: One in the north west, and another one in the south east.
Since these two areas do not move much, it should be easy to identify them as coherent sets.
In the second example, it is less obvious what the optimal partition should be.
Both discrete densities $\mu$ and $\nu$ are normalized such that $\sum_i \mu_i = \sum_j \nu_j = 1$.
Now, we can apply our proposed procedure from Section~\ref{sec:discrete}:
First, we compute the regularized optimal transport plan $\hat{\pi}_\varepsilon$ with regularization parameter $\varepsilon=0.01$ using the Python Optimal Transport package (POT)\footnote{Code available at \url{https://github.com/rflamary/POT} (accessed: 26.06.2020)}, where the distance between two neighboring pixels is $1$.
Then, we compute the singular vectors $u_2$ and $v_2$ of the matrix \smash{$\sighalf_\mu \hat{\pi}_\varepsilon \sighalf_\nu$} as well as the optimal partition vectors $\hat{f} = \sighalf_\mu u_2$ and $\hat{g} = \sighalf_\nu v_2$.
Numerically, we observe that choosing $\varepsilon$ small produces partition functions that are almost constant on the respective parts with a sharp transition between them, whereas very large $\varepsilon$ yields approximately affine partition functions, which is clearly not what we want.
In any case, $\varepsilon$ has to be chosen large enough such that the Sinkhorn algorithm converges in acceptable time.
\renewcommand\curfolder{img/balanced_1e-2_W2_1}
\renewcommand\curimgwidth{0.17\textwidth}
\renewcommand\curbarwidth{0.09\textwidth}
\renewcommand\curboxwidth{0.18\textwidth}
\begin{figure}[tbp]
	\centering
	\subfloat[$\hat{f} = \Sigma^{-1/2}_\mu u_2$.]{
	\makebox[\curboxwidth][c]{\includegraphics[align=c, width=\curimgwidth]{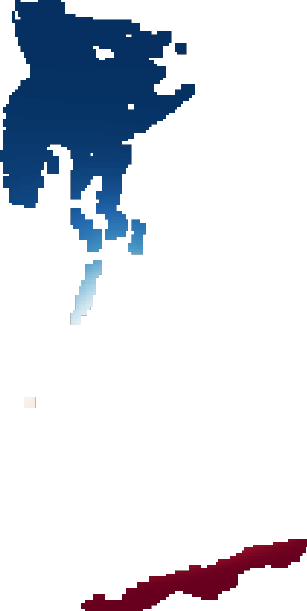}}}
	\,
	\subfloat[$\hat{g} = \Sigma^{-1/2}_\nu v_2$.]{\makebox[\curboxwidth][c]{\includegraphics[align=c, width=\curimgwidth]{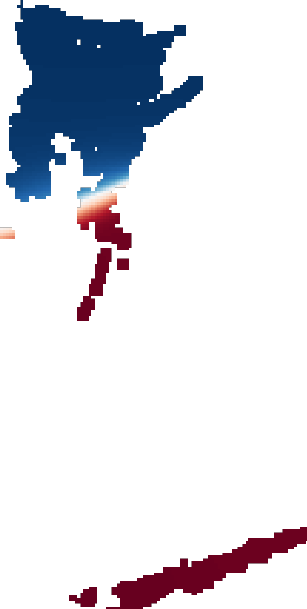}}}
	\,
	\subfloat{\includegraphics[align=c, width=\curbarwidth]{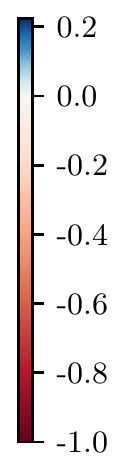}}
	\addtocounter{subfigure}{-1}
	\,
	\subfloat[Classifier $\hat{f} \lesseqgtr 0$.]{\makebox[\curboxwidth][c]{\includegraphics[align=c, width=\curimgwidth]{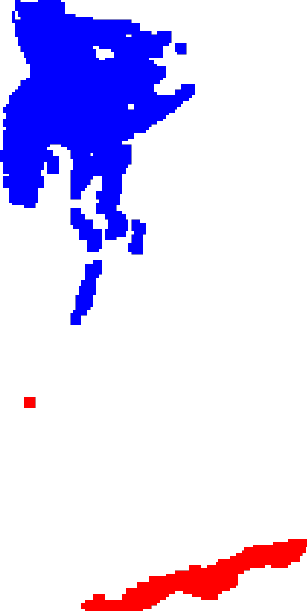}}}
	\,
	\subfloat[Classifier $\hat{g} \lesseqgtr 0$.]{\makebox[\curboxwidth][c]{\includegraphics[align=c, width=\curimgwidth]{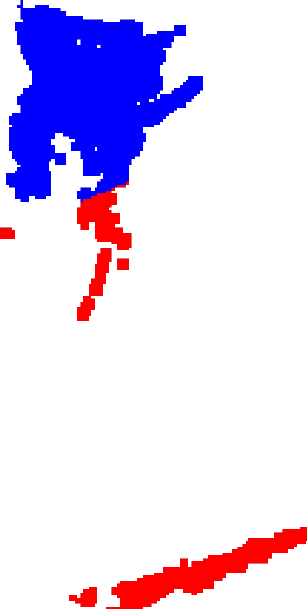}}}
	\caption{Computed coherent sets for the precipitation densities in Figs.~\ref{subfig:rain_1_t0}--\ref{subfig:rain_1_t1} using regularized OT.}
	\label{fig:opt_partition_1}
	
	
	\vspace{.5cm}

    \renewcommand\curfolder{img/lambd1e-2_reg1_unnormalized_W2_1}

	\centering
	\subfloat[$\hat{f} = \Sigma^{-1/2}_\mu u_2$.]{\makebox[\curboxwidth][c]{\includegraphics[align=c, width=\curimgwidth]{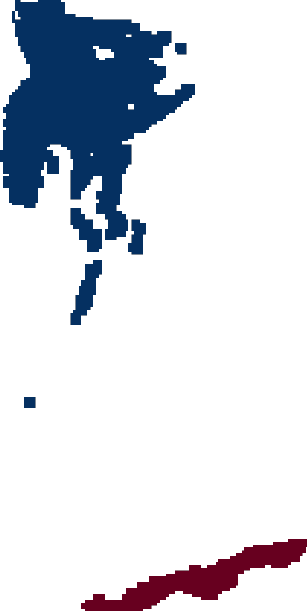}}}
	\,
	\subfloat[$\hat{g} = \Sigma^{-1/2}_\nu v_2$.]{\makebox[\curboxwidth][c]{\includegraphics[align=c, width=\curimgwidth]{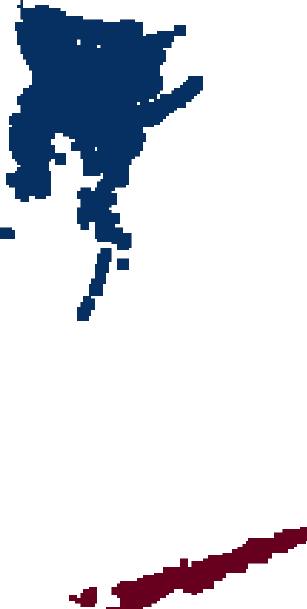}}}
	\,
	\subfloat{\includegraphics[align=c, width=\curbarwidth]{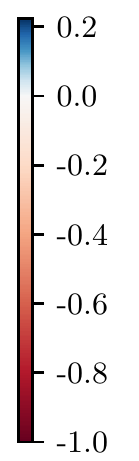}}
	\addtocounter{subfigure}{-1}
	\,
	\subfloat[Classifier $\hat{f} \lesseqgtr 0$.]{\makebox[\curboxwidth][c]{\includegraphics[align=c, width=\curimgwidth]{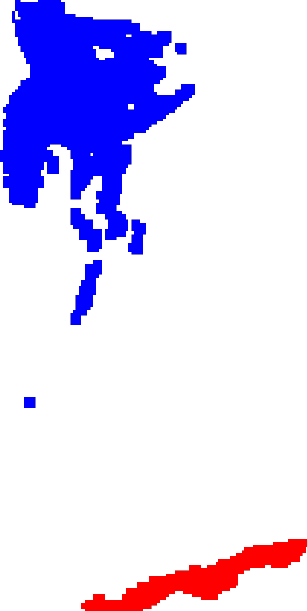}}}
	\,
	\subfloat[Classifier $\hat{g} \lesseqgtr 0$.]{\makebox[\curboxwidth][c]{\includegraphics[align=c, width=\curimgwidth]{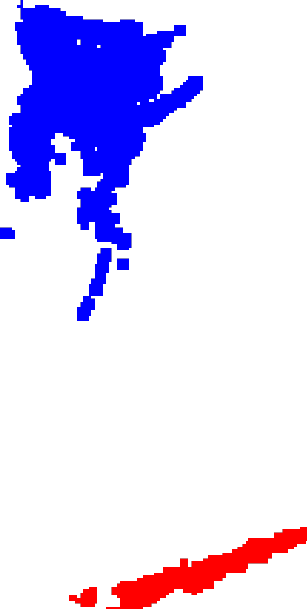}}}
	\caption{Computed coherent sets for the precipitation densities in Figs.~\ref{subfig:rain_1_t0}--\ref{subfig:rain_1_t1} using unbalanced regularized OT.}
	\label{fig:unbalanced_1}
	
	
	\vspace{.5cm}
	
    \renewcommand\curfolder{img/lambd1e-2_reg1_unnormalized_W2_2}
	\centering
	\subfloat[$\hat{f} = \Sigma^{-1/2}_\mu u_2$.]{\makebox[\curboxwidth][c]{\includegraphics[align=c, width=\curimgwidth]{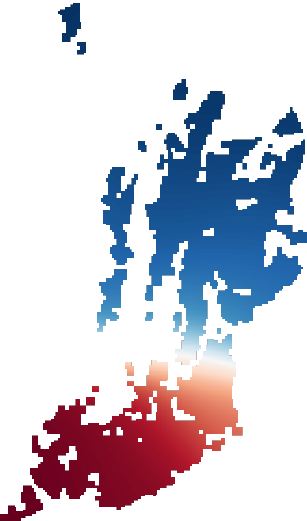}}}
	\,
	\subfloat[$\hat{g} = \Sigma^{-1/2}_\nu v_2$.]{\makebox[\curboxwidth][c]{\includegraphics[align=c, width=\curimgwidth]{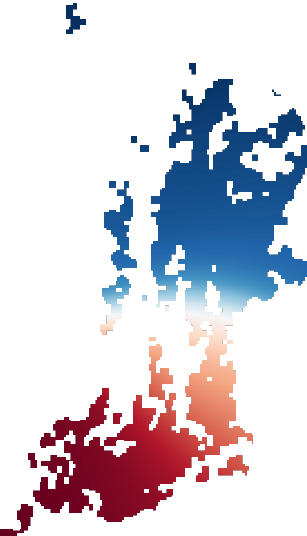}}}
	\,
	\subfloat{\includegraphics[align=c, width=\curbarwidth]{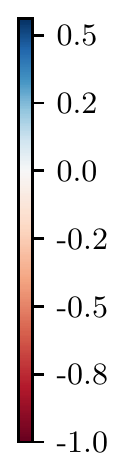}}
	\addtocounter{subfigure}{-1}
	\,
	\subfloat[Classifier $\hat{f} \lesseqgtr 0$.]{\makebox[\curboxwidth][c]{\includegraphics[align=c, width=\curimgwidth]{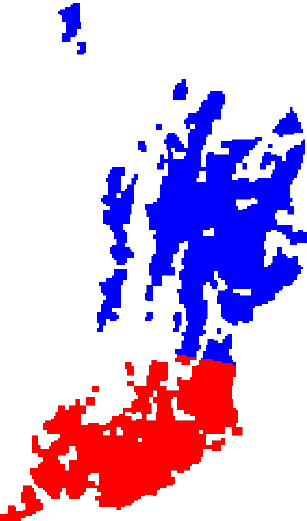}}}
	\,
	\subfloat[Classifier $\hat{g} \lesseqgtr 0$.]{\makebox[\curboxwidth][c]{\includegraphics[align=c, width=\curimgwidth]{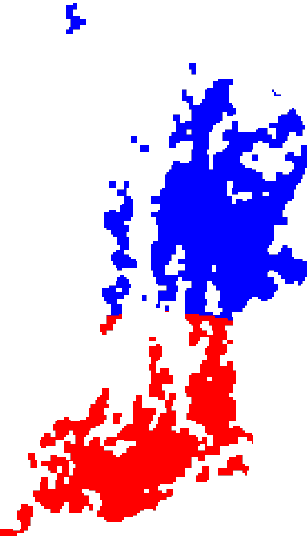}}}
	\caption{Computed coherent sets for the precipitation densities in Figs.~\ref{subfig:rain_2_t0}--\ref{subfig:rain_2_t1} using unbalanced regularized OT.}
	\label{fig:unbalanced_2}
\end{figure}

Results are shown in Fig.~\ref{fig:opt_partition_1}.
Note that the dark-blue pixels in Fig.~\ref{fig:data_clouds} are only included to indicate the domain of $\mu$ and $\nu$ and not part of the computation.
Hence, they are discarded in Figs.~\ref{fig:opt_partition_1}--\ref{fig:unbalanced_2}.
We observe that the method works roughly as expected:
As shown in our theory part, the first singular value is one and belongs to the singular pair $(1_{\X}, 1_{\Y})$.
Further, the two precipitation areas are indicated by differing signs in the partition vectors corresponding to the second singular value.
However, taking a closer look, splitting the precipitation areas (or ``fuzzy classification'') does not work perfectly.
In Fig.~\ref{fig:opt_partition_1}, some parts of the bigger ``cloud'' (as we refer to precipitation areas from now on) have values greater than zero, although the rest has negative values.
Ceonsequently, some parts of the upper cloud are classified as ``red'', i.e., as part of the lower cloud.
This issue is even more pronounced for the included hard ``classification'' according to the criteria $\hat{f} \geq 0$ and $\hat{g} \geq 0$.
Note that simple thresholding was sufficient for our purposes, but more sophisticated approaches such as fuzzy $c$-means (which will be used in the next subsection) or sparsity-promoting clustering, see~\citet{froyland2019sparse}, are also applicable.

Next, we want to resolve the mentioned issue of wrong classification.
Of course, it is inaccurate to model precipitation densities as transported mass particles, since they can come down as rain and simply disappear.
Hence, the misclassification is not surprising as the OT plan needs to transport this disappearing mass somewhere else.
Indeed, an examination of the transport plan reveals that mass preservation does not hold in the two clouds, thus they cannot fulfill the coherence condition $L 1_{X_k} = 1_{Y_k}$ (see Section~\ref{sec:segm}) and the transport needs to shift some of the mass from the upper cloud to the lower one.

To compensate for this effect, we propose to use unbalanced regularized OT, relaxing the mass conservation condition on the marginals, see Section~\ref{sec:OT}.
Noteworthy, this modification does not introduce any major computational overhead.
Again, we use $\varepsilon=0.01$ and choose $\kappa=1$, see~\eqref{eq:unbalanced_entropic_ot}. As mentioned by \citet{SFVTP19}, the parameter $\kappa$ intuitively corresponds to a choice of radius for which, when exceeded, it is cheaper to produce or destroy mass instead of transporting it between the two clouds.
Hence, we should choose $\kappa$ small enough to have an effect compared to balanced OT, but large enough so that the marginals of the resulting transport plan are still close to the given $\mu$ and $\nu$.

The results for the data from Figs.~\ref{subfig:rain_1_t0}--\ref{subfig:rain_1_t1} and Figs.~\ref{subfig:rain_2_t0}--\ref{subfig:rain_2_t1} 
are displayed in Figs.~\ref{fig:unbalanced_1} and~\ref{fig:unbalanced_2}, respectively.
As expected, the first singular pair is still given by $(1_{\X}, 1_{\Y})$. 
Now, the area classification in Figs.~\ref{subfig:rain_1_t0}--\ref{subfig:rain_1_t1} is more or less perfect and hence also the second singular value is almost one.
This might not be surprising, as the areas are well separated and moving mass between them is quite expensive.
Compared to standard regularized transport, using unbalanced regularized transport for the data from Figs.~\ref{subfig:rain_2_t0}--\ref{subfig:rain_2_t1} does not change the partition much.
Overall, the results based on unbalanced regularized OT look very promising and hence we use such plans in all further experiments.
Note that a comparison with the coherent set methods mentioned in Section~\ref{sec:intro} is not possible, since these require a known dynamic.

\subsection{Particles Moving in a Potential} \label{subsec:wells}
In this subsection, we discuss the example of \citet[Sec.~6.2, p.~18]{koltai2018optimal} in a slightly modified form.
The dynamical system under consideration consists of particles moving according to standard Brownian motion with a drift term induced by a potential and a rotating force.
More precisely, our particle trajectories are solutions to the stochastic differential equation \[
\dx x_t = (F_r + F_c)(x_t) \dx t + \sqrt{2\beta^{-1}}\dx w_t,
\]
where $w_t$ denotes standard Brownian motion, $F_r(x) = -\nabla W(x)$ is the force coming from the $3$-well potential
\[
W(x) = \cos(3\varphi) + 10(r-1)^2,\quad \text{where } x= \binom{r\cos(\varphi)}{r\sin(\varphi)}
\]
and $F_c$ is a circular driving force in clock-wise direction, given by
\[
F_c(x) \coloneqq \mathrm{e}^{-\beta W(x)}
\begin{pmatrix}
0 & 1 \\
-1 & 0
\end{pmatrix}
x.
\]
In statistical physics, $\beta$ is called \emph{inverse temperature}, which is chosen as $2.0$ in this experiment. Note that $F_c$ is strongest in the valley of each well.
Overall, the particles tend to remain in a well for some time and occasionally hop to another one due to their diffusive motion and the circular force.
We expect the particle distribution to converge to an equilibrium with modes at the potential minima, slightly rotated in clock-wise direction.

First, we create $1000$ initial particle positions for the measure $\mu$ as follows:
We sample a single-particle-trajectory with $50000$ steps of time length~$0.01$ using an Euler--Maruyama-scheme starting from $(1,0)$, where the circular driving force $F_c$ is neglected.
Then, we draw $1000$ trajectory points uniformly at random from these points without replacement. The final discrete measure $\nu$ is simulated by
the numerical trajectories for each one of these $1000$ initial particle positions using $300$ Euler--Maruyama-steps of length~$0.01$, where $F_c$ is included now. This yields $1000$ pairs of starting and ending position $(x_0^i, x_1^i)$ displayed in Fig.~\ref{fig:data_wells}.
To visualize the trajectory of individual particles, we colored each particle in Fig.~\ref{subfig:wells_data_1} in the same color as in Fig.~\ref{subfig:wells_data_0}. Further, in Fig.~\ref{subfig:wells_lines}, we drew connecting lines between initial and terminal particle positions.

Next, we apply Algorithm \ref{alg:u-r-ot} to
\smash{$\mu = \frac{1}{n}\sum_{i=1}^n \delta_{x_0^i}$ and $\nu = \frac{1}{n}\sum_{i=1}^n \delta_{x_1^i}$} with $\kappa=1$.
The regularization parameter $\varepsilon$ is chosen so that the corresponding kernel $\exp(-c/\varepsilon)$ (see \eqref{sinkhorn_kernel}) has standard deviation one third of the particles mean distance at $t=0$.
We write $f_k = \sighalf_\mu u_k$ and $g_k = \sighalf_\nu v_k$ for the partition vectors belonging to the decreasingly ordered singular values $\sigma_k$, where $u_k$ and $v_k$ are the corresponding left and right singular vectors.
The obtained partition vector pairs for $k=2,3$ are shown in Fig.~\ref{fig:partition_wells}.

Using the information from $f_k, g_k$, $k=2, 3$, we can embed every point $x_t^i$ into $\R^2$ via $\Phi \colon x_0^i \mapsto ((f_2, f_3))_i$, $x_1^i \mapsto ((g_2, g_3))_i$.
Now, we partition the data using the fuzzy $c$-means algorithm\footnote{Implementation available at \url{https://github.com/omadson/fuzzy-c-means} (accessed: 26.06.2020)} as described in Sec.~\ref{sec:segm} on this embedded data for three clusters, see also \citet{SS2012}.
In contrast to applying fuzzy $c$-means directly on the individual snapshots at $t=0, 1$, we naturally obtain correspondences between the clusters at the different time steps, where OT serves as a proxy for the underlying dynamics.
The obtained hard clusters and fuzzy membership values are displayed in Fig.~\ref{fig:cmeans_wells}.

Note that applying the method in \citet{froyland2013analytic} under usage of the particle label information cannot yield coherent sets for such strong particle mixing.
Other methods such as the finite-time Lyapunov exponent \citep{shadden2005definition} also only work for short timespans.
On the other hand, comparing this with our method, it is clear from Fig.~\ref{fig:data_wells} that the OT assignment will be very different from the ground truth particle correspondences.
Thus, it aims for a partition based on the ``macroscopic'', ensemble density level rather than on the ``microscopic'' level of individual, distinguishable particles.
Indeed, we observe that the coherent sets are three denser blobs due to the energy landscape.
Since there was no circular force present during the construction of $\mu$, there is a slight offset between the initial and stationary distributions, which is captured by the coherent sets.
As expected, the likeliness of being in a cluster decreases if a point is close to the boundaries of the well.
Consequently, there is some uncertainty about the exact cluster boundaries.
\renewcommand\curfolder{img/wells}
\renewcommand\curwidth{0.27\textwidth}
\begin{figure}[htbp]
	\centering
	\subfloat[Data at time $t=0$.\label{subfig:wells_data_0}]{\includegraphics[align=c, width=\curwidth]{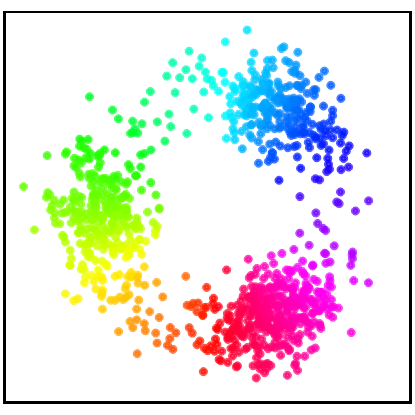}}
	\,
	\subfloat[Data at time $t=1$.\label{subfig:wells_data_1}]{\includegraphics[align=c, width=\curwidth]{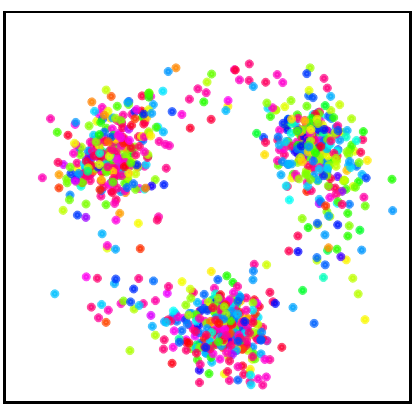}}
	\,
	\subfloat[Movement of particles.\label{subfig:wells_lines}]{\includegraphics[align=c, width=\curwidth]{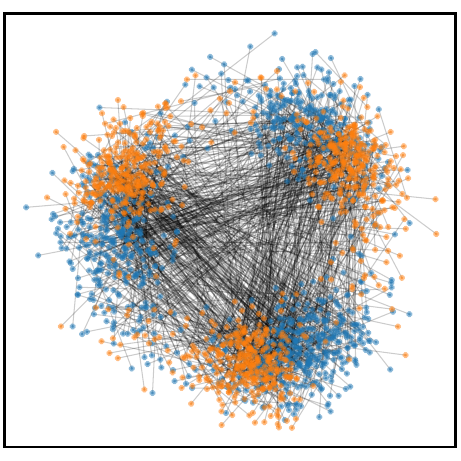}}
	\caption{Particles moving in a potential with circular driving force.
	The colors in Figs.~(a)--(b) indicate the angular coordinates at time $t=0$, illustrating the particle mixing.}
	\label{fig:data_wells}

    \vspace{.5cm}
	
	\renewcommand\curwidth{0.19\textwidth}
	\renewcommand\curbarwidth{0.07\textwidth}
	\subfloat[$f_2 = \Sigma^{-1/2}_\mu u_2$.]{\includegraphics[align=c, width=\curwidth]{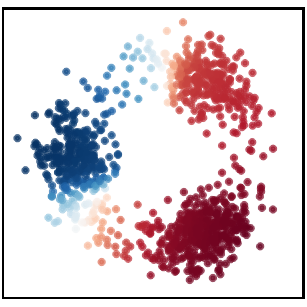}}
	\,
	\subfloat[$g_2 = \Sigma^{-1/2}_\nu v_2$.]{\includegraphics[align=c, width=\curwidth]{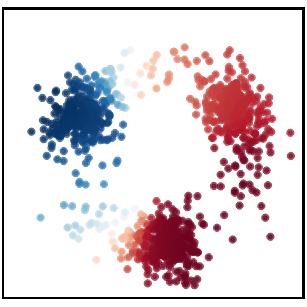}}
	\,
	\subfloat{\includegraphics[align=c, width=\curbarwidth]{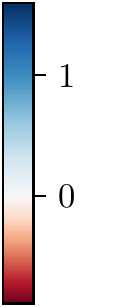}}
	\addtocounter{subfigure}{-1}
	\,
	\subfloat[$f_3 = \Sigma^{-1/2}_\mu u_3$.]{\includegraphics[align=c, width=\curwidth]{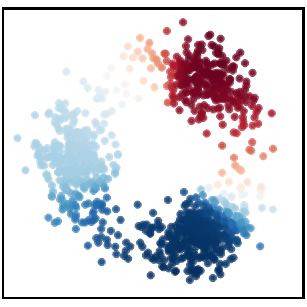}}
	\,
	\subfloat[$g_3 = \Sigma^{-1/2}_\nu v_3$.]{\includegraphics[align=c, width=\curwidth]{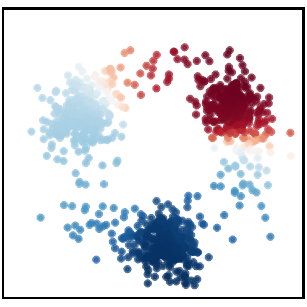}}
	\,
	\subfloat{\includegraphics[align=c, width=\curbarwidth]{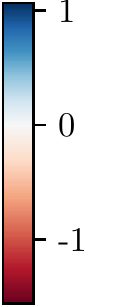}}
	\addtocounter{subfigure}{-1}
	\caption{First and second partition vectors for data displayed in Fig.~\ref{fig:data_wells}.}
	\label{fig:partition_wells}
		
	\vspace{.5cm}
	
	\renewcommand\curwidth{0.19\textwidth}
	\renewcommand\curbarwidth{0.07\textwidth}
	\subfloat[Hard, $t=0$.]{\includegraphics[align=c, width=\curwidth]{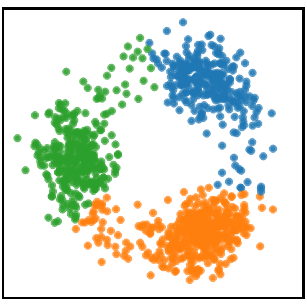}}
	\,
	\subfloat[Hard, $t=1$.]{\includegraphics[align=c, width=\curwidth]{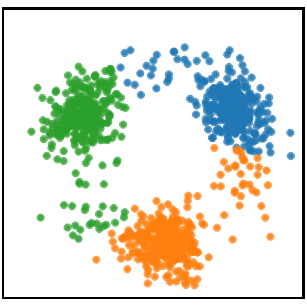}}
	\,
	\subfloat{\includegraphics[align=c, width=\curbarwidth]{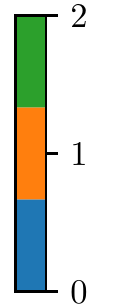}}
	\addtocounter{subfigure}{-1}
	\,
	\subfloat[Cluster $0$, $t=0$.]{\includegraphics[align=c, width=\curwidth]{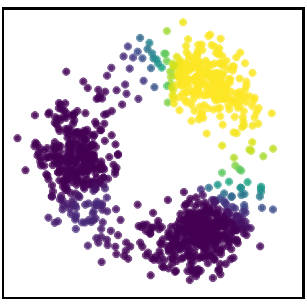}}
	\,
	\subfloat[Cluster $0$, $t=1$.]{\includegraphics[align=c, width=\curwidth]{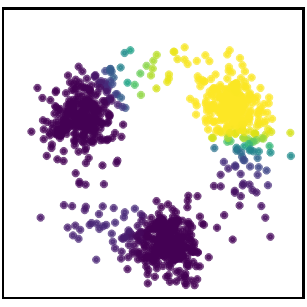}}
	\,
	\subfloat{\includegraphics[align=c, width=\curbarwidth]{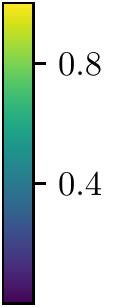}}
	\addtocounter{subfigure}{-1}
	\par\smallskip
	\subfloat[Cluster $1$, $t=0$.]{\includegraphics[align=c, width=\curwidth]{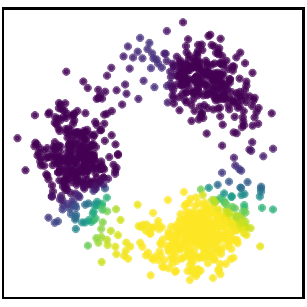}}
	\,
	\subfloat[Cluster $1$, $t=1$.]{\includegraphics[align=c, width=\curwidth]{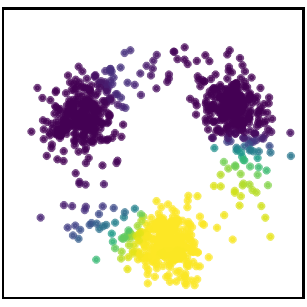}}
	\,
	\subfloat{\includegraphics[align=c, width=\curbarwidth]{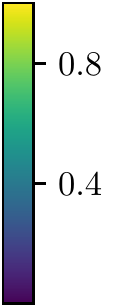}}
	\addtocounter{subfigure}{-1}
	\,
	\subfloat[Cluster $2$, $t=0$.]{\includegraphics[align=c, width=\curwidth]{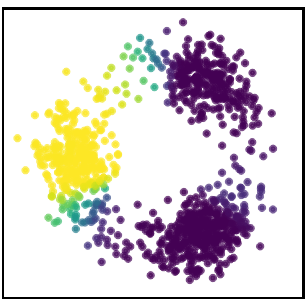}}
	\,
	\subfloat[Cluster $2$, $t=1$.]{\includegraphics[align=c, width=\curwidth]{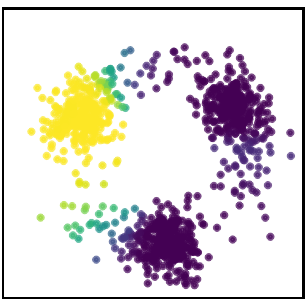}}
	\,
	\subfloat{\includegraphics[align=c, width=\curbarwidth]{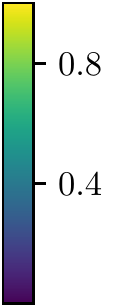}}
	\caption{Results of hard classification and fuzzy $c$-means.
	The color scheme in (c)--(h) encodes the likeliness that a point belongs to the cluster $j$ with $j=0, 1, 2$.}
	\label{fig:cmeans_wells}
\end{figure}

\subsection{Particles with Pairwise Lennart--Jones Potentials}\label{subsec:lennart-jones}

In the last example, we aim for a segmentation of particle groups with a slightly more realistic data set, consisting of $200$ particle trajectories created with the molecular dynamics simulation software LAMMPS \citep{lammps}.
In our simulation, the particles interact with each other in terms of a pairwise Lennart--Jones potential with cutoff, essentially repelling each other in close proximity but attracting each other otherwise, such that there is some optimal energy-minimizing pairwise distance \citep{rapaport2004art}.
Given some initial velocity for the particles, they start to stick to each other over time and slowly form groups, which in turn connect to larger groups and so on.
We take two snapshots of the simulation showing some group formation, see Fig.~\ref{fig:data_lj}.
As the domain is the two-dimensional torus, particles leaving the domain on one side come back in from the opposite side in the visualization.
\footnote{The script for generating the trajectories is available at the blog post under \url{http://nznano.blogspot.com/2017/11/molecular-dynamics-in-python.html\#Implementation-in-LAMMPS} (accessed: 26.06.2020)}

Then, we apply our computational scheme with the same parameter choices as in Subsection~\ref{subsec:wells} to obtain the clusters.
The computed segmentation vectors and the induced fuzzy clustering are depicted in Figs.~\ref{fig:partition_lj} and~\ref{fig:cmeans_lj}, respectively.
From a visual point of view, the clustering scheme ``correctly'' detects the stable group on the left, roughly three connecting groups in the middle and several smaller connecting groups on the right, independent of the mixing of several individual particles.

As we obtained the data by a molecular dynamics simulation, we can access the ``ground truth'' particle labels and track the position of every particle forward and backward in time.
Since the particle mixing in this example is not as strong as in Section~\ref{subsec:wells}, it is natural to ask whether the hard cluster labels of particle $i$ are the same at $t=0, 1$.
We observe that the cluster labels agree in $176$ of the total $200$ cases, that is, around $84\%$ of the particles retain their cluster label in time.
Furthermore, we may compare the hard cluster label of each particle $i$ at time $t=0$ with the one that particle $i$ is assigned at time $t=1$ in two scatter plots, where the position in the plot of each particle is fixed and the two different colorings indicate the hard cluster labels for $t=0$ and $t=1$ (that is, corresponding to the left or right singular vectors), respectively.
The results are shown in Fig.~\ref{fig:ground_truth_lj}. 
\renewcommand\curfolder{img/lennart_jones}
\renewcommand\curwidth{0.27\textwidth}
\begin{figure}[htbp]
	\centering
	\subfloat[Data at time $t=0$.]{\includegraphics[align=c, width=\curwidth]{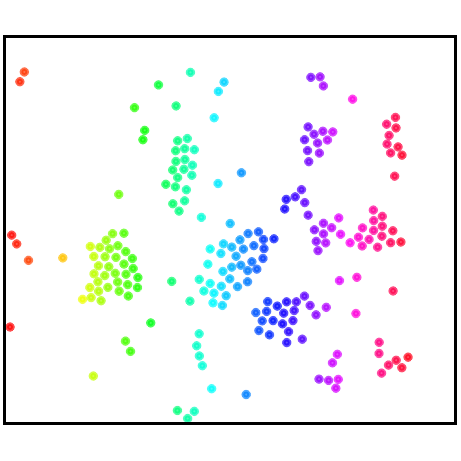}}
	\,
	\subfloat[Data at time $t=1$.]{\includegraphics[align=c, width=\curwidth]{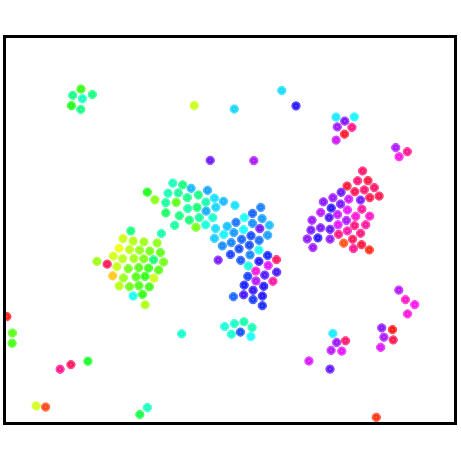}}
	\caption{Particles moving in a potential with circular driving force.
	The colors indicate the horizontal coordinate at $t=0$, again illustrating the particle mixing.}
	\label{fig:data_lj}

    \vspace{.5cm}
	
	\renewcommand\curwidth{0.19\textwidth}
	\renewcommand\curbarwidth{0.07\textwidth}
	\subfloat[$f_2 = \Sigma^{-1/2}_\mu u_2$.]{\includegraphics[align=c, width=\curwidth]{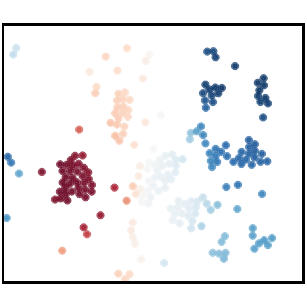}}
	\,
	\subfloat[$g_2 = \Sigma^{-1/2}_\nu v_2$.]{\includegraphics[align=c, width=\curwidth]{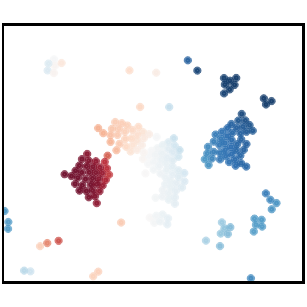}}
	\,
	\subfloat{\includegraphics[align=c, width=\curbarwidth]{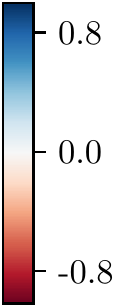}}
	\addtocounter{subfigure}{-1}
	\,
	\subfloat[$f_3 = \Sigma^{-1/2}_\mu u_3$.]{\includegraphics[align=c, width=\curwidth]{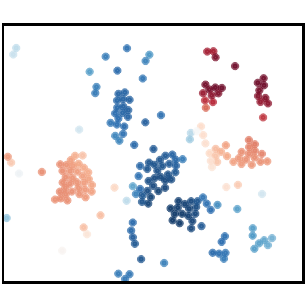}}
	\,
	\subfloat[$g_3 = \Sigma^{-1/2}_\nu v_3$.]{\includegraphics[align=c, width=\curwidth]{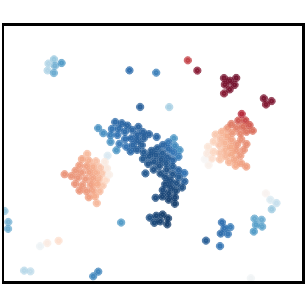}}
	\,
	\subfloat{\includegraphics[align=c, width=\curbarwidth]{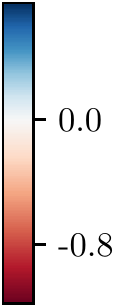}}
	\addtocounter{subfigure}{-1}
	\caption{First and second partition vectors for data displayed in Fig.~\ref{fig:data_lj}.}
	\label{fig:partition_lj}
		
	\vspace{.5cm}
	
	\renewcommand\curwidth{0.19\textwidth}
	\renewcommand\curbarwidth{0.07\textwidth}
	\subfloat[Hard, $t=0$.]{\includegraphics[align=c, width=\curwidth]{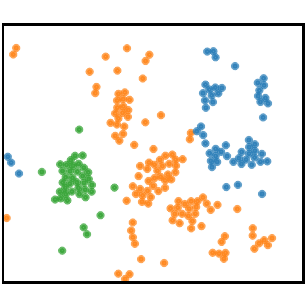}}
	\,
	\subfloat[Hard, $t=1$.]{\includegraphics[align=c, width=\curwidth]{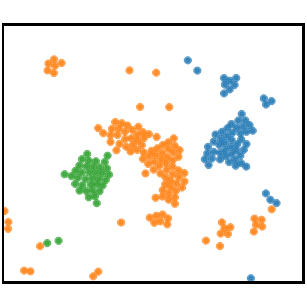}}
	\,
	\subfloat{\includegraphics[align=c, width=\curbarwidth]{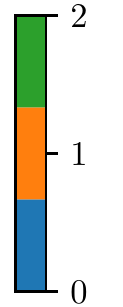}}
	\addtocounter{subfigure}{-1}
	\,
	\subfloat[Cluster $0$, $t=0$.]{\includegraphics[align=c, width=\curwidth]{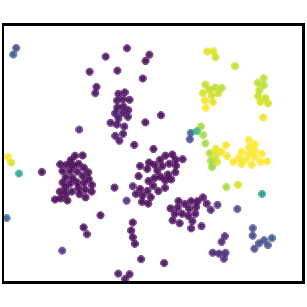}}
	\,
	\subfloat[Cluster $0$, $t=1$.]{\includegraphics[align=c, width=\curwidth]{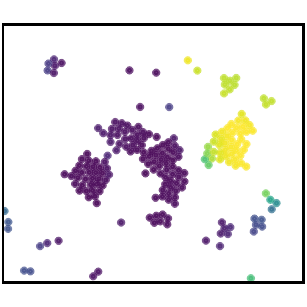}}
	\,
	\subfloat{\includegraphics[align=c, width=\curbarwidth]{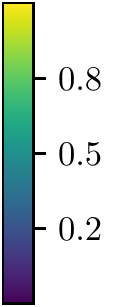}}
	\addtocounter{subfigure}{-1}
	\par
	\subfloat[Cluster $1$, $t=0$.]{\includegraphics[align=c, width=\curwidth]{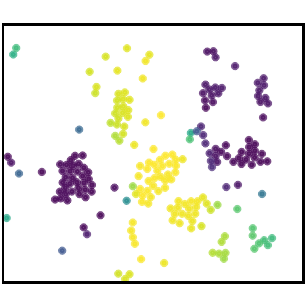}}
	\,
	\subfloat[Cluster $1$, $t=1$.]{\includegraphics[align=c, width=\curwidth]{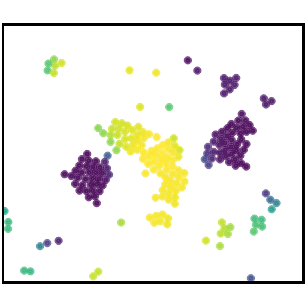}}
	\,
	\subfloat{\includegraphics[align=c, width=\curbarwidth]{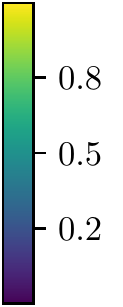}}
	\addtocounter{subfigure}{-1}
	\,
	\subfloat[Cluster $2$, $t=0$.]{\includegraphics[align=c, width=\curwidth]{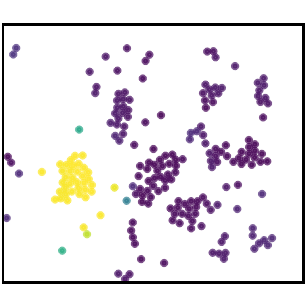}}
	\,
	\subfloat[Cluster $2$, $t=1$.]{\includegraphics[align=c, width=\curwidth]{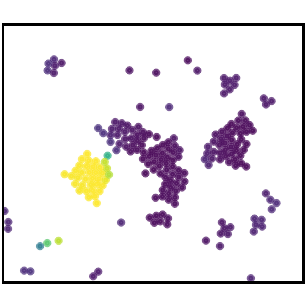}}
	\,
	\subfloat{\includegraphics[align=c, width=\curbarwidth]{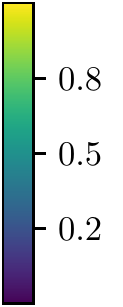}}
	\caption{Results of hard classification and fuzzy $c$-means.
	The color scheme in (c)--(h) encodes the likeliness that a point belongs to the corresponding cluster $j$ with $j=0, 1, 2$.}
	\label{fig:cmeans_lj}
\end{figure}
\begin{figure}
    \centering
	\renewcommand\curwidth{0.23\textwidth}
	\subfloat[Positions for $t=0$, cluster labels for $t=0$.]{\includegraphics[align=c, width=\curwidth]{\curfolder/hard_t0.pdf}}
	\,
	\subfloat[Positions for $t=0$, cluster labels for $t=1$.]{\includegraphics[align=c, width=\curwidth]{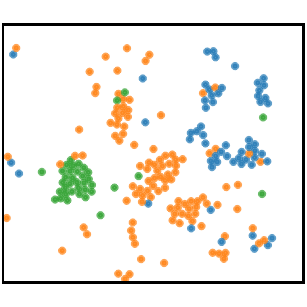}}
	\,
	\subfloat[Positions for $t=1$, cluster labels for $t=1$.]{\includegraphics[align=c, width=\curwidth]{\curfolder/hard_t1.pdf}}
	\,
	\subfloat[Positions for $t=1$, cluster labels for $t=0$.]{\includegraphics[align=c, width=\curwidth]{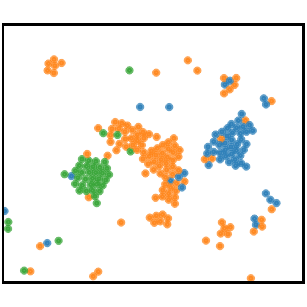}}
	\caption{Comparison of hard cluster labels corresponding to left and right singular vectors, for particle positions at both $t=0$ and $t=1$, respectively. Here we have used the ground truth particle labels $i$.}
	\label{fig:ground_truth_lj}
\end{figure}

\subsection{Particle Tracking with Concatenated Transfer Operators}\label{sec:Ex4}
Here, we discuss an example also analyzed by \citet{froyland2014almost} and \citet[Sec.~V.A]{banisch2017understanding}.
Consider the non-autonomous system
\begin{equation}\label{eq:ODE_Ex4}
    \begin{aligned}
    \frac{\dx x}{\dx t} &= -\pi A \sin\bigl(\pi f(t, x)\bigr)\cos(\pi y) \\
    \frac{\dx y}{\dx t} &= \pi A \cos\bigl(\pi f(t, x)\bigr)\sin(\pi y) \frac{\dx f}{\dx x}(t, x)
    \end{aligned}
\end{equation}
with $f(x, t) = \alpha \sin(\omega t)x^2 + (1-2\alpha \sin(\omega t))x$ and parameters $A=0.25$, $\alpha = 0.25$ and $\omega = 2\pi$.
This system describes two counter-rotating gyres, where the vertical boundary between them oscillates periodically.
Moreover, it preserves the Lebesgue measure on $\X=\Y=[0,1]\times[0,2]$.

We initialize $n = 450$ particles on an equispaced rectangular grid.
Then, we compute their trajectories by solving \eqref{eq:ODE_Ex4} with $500$ time steps of length $\Delta t=0.02$, i.e., on the time interval~$[0,10]$.
This yields $N=501$ measures $\mu_t=\frac 1 n \sum_{i=1}^n \delta x_t^i$, $t=0,\ldots,N$.
Note that $\Delta t$ is small enough to recover the ground truth particle correspondences in most cases.
Based on the entropy regularized optimal transport plans $\hat \pi_\varepsilon$ between $\mu_t$ and $\mu_{t+1}$ with regularization parameter $\varepsilon=10^{-3}$, we construct corresponding transfer operators $L_{\varepsilon, t}\colon L_2(\R^n)\to L_2(\R^n)$ as $L_{\varepsilon, t} = \sighalf_{\mu_t} \hat \pi_\varepsilon^T \sighalf_{\mu_{t+1}}$.
Again, the entropy regularization readily introduces the required diffusion.
Efficient implementations of the Sinkhorn algorithm can be achieved using, e.g., multiscale schemes \citep{schmitzer19} or parallelization on GPUs \citep{C2013}.
Next, we compute the concatenated transfer operator
$L_\varepsilon = \smash{\prod_{t=0}^{N-1}} L_{\varepsilon, t}$,
for which its matrix is shown in Fig.~\ref{fig:trans_mat}.
Its block-diagonal structure already indicates coherent sets.
Finally, the partitions corresponding to the second and third singular vectors $f_i$, $g_i$, $i=1,2$,
of its SVD are shown in Fig.~\ref{fig:partition_dg}.
The corresponding singular values are $\sigma_2 \approx 0.71$ and $\sigma_3 \approx 0.35$, respectively.
For a hard clustering using fuzzy $c$-means we refer to Fig.~\ref{fig:double_gyre_clustering} in the introduction.

Particle transitions between the left and the right half of the domain are very rare, as indicated by the optimal partition.
Further, the third singular vectors illustrate that the particles, which move in closed curves around the respective gyre cores, take a long time to transition from the gyre centers to their boundaries or vice versa. 
In summary, this example illustrates how our method can be used to compute coherent sets for flows with unlabeled particles.
Here, OT is used to track them through observations of subsequent timesteps; see also Particle Image Velocimetry \citep{saumier2015piv}.

\renewcommand\curfolder{img/double_gyre}
\begin{figure}[htbp]
	\centering
	
    \renewcommand\curwidth{0.4\textwidth}
	\includegraphics[align=c, width=\curwidth]{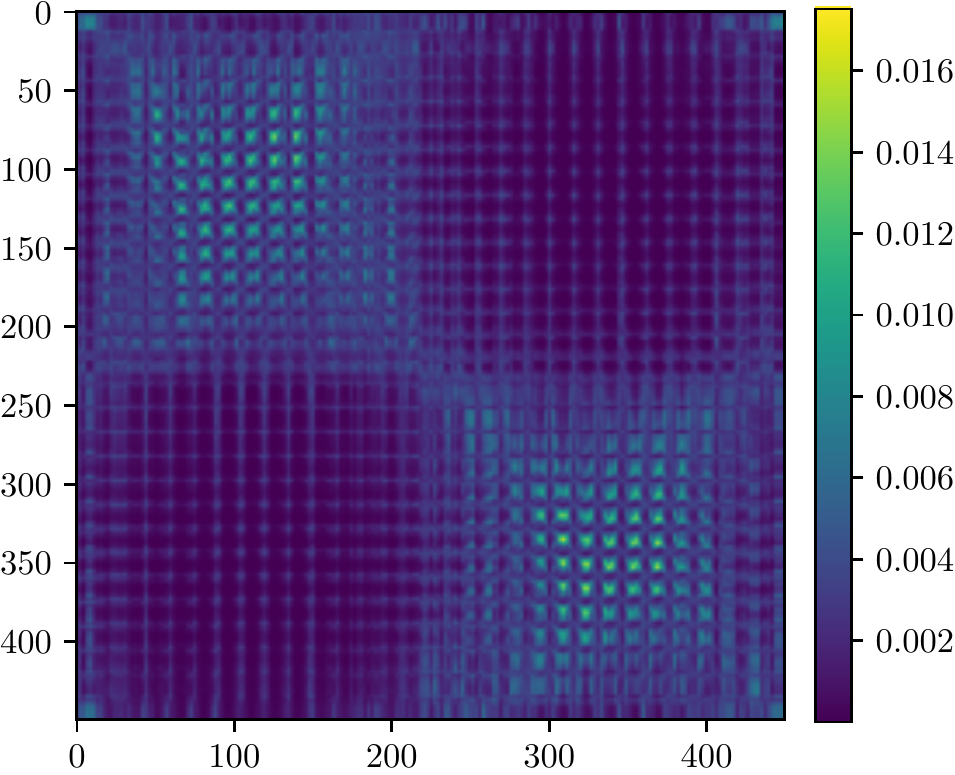}
	\caption{Matrix of concatenated transfer operator $L_\varepsilon$. The rows and columns are ordered according to the horizontal and then vertical coordinates of the corresponding particles in the initial configuration.}
	\label{fig:trans_mat}
	
	\vspace{.5cm}
	
    \renewcommand\curwidth{0.35\textwidth}
	\renewcommand\curbarwidth{0.07\textwidth}
	\subfloat[$f_2 = \Sigma^{-1/2}_\mu u_2$.]{\includegraphics[align=c, width=\curwidth]{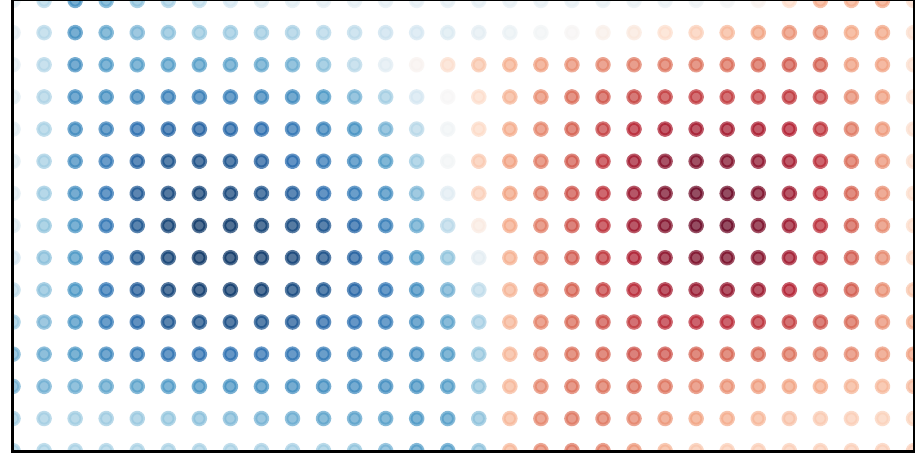}}
	\,
	\subfloat[$g_2 = \Sigma^{-1/2}_\nu v_2$.]{\includegraphics[align=c, width=\curwidth]{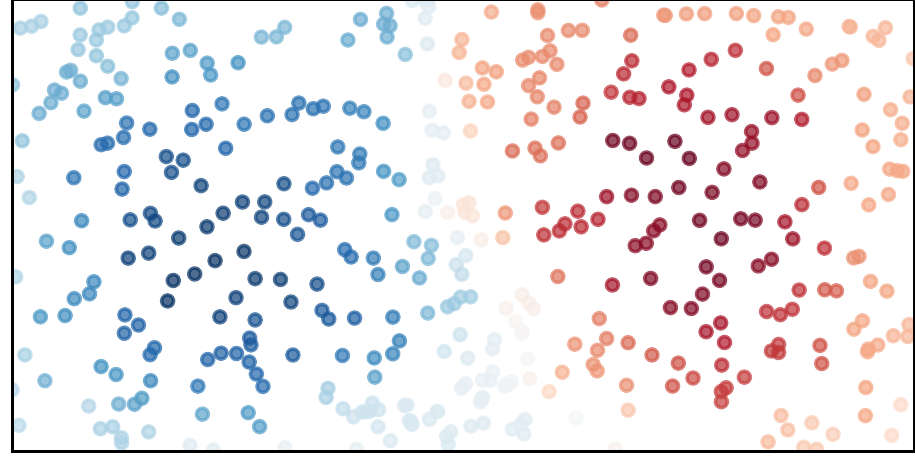}}
	\,
	\subfloat{\includegraphics[align=c, width=\curbarwidth]{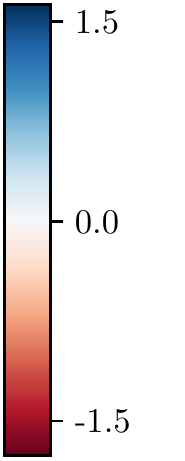}}
	\addtocounter{subfigure}{-1}
	
	\vspace{.5cm}
	
	\subfloat[$f_3 = \Sigma^{-1/2}_\mu u_3$.]{\includegraphics[align=c, width=\curwidth]{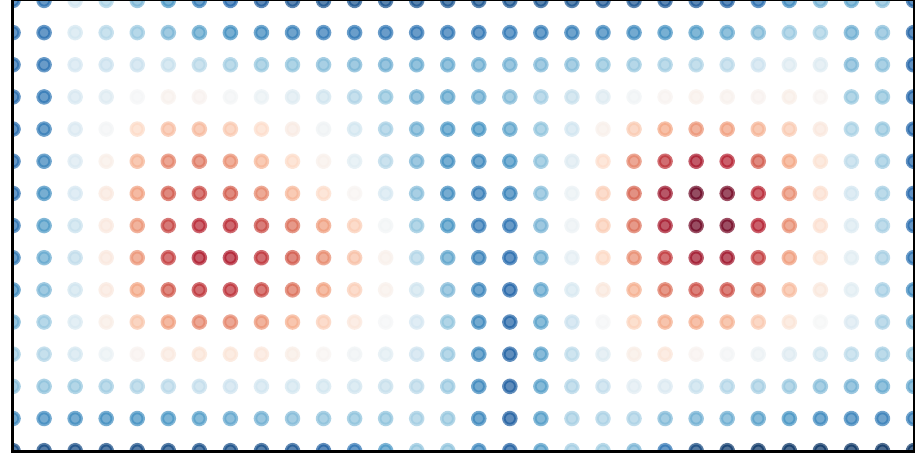}}
	\,
	\subfloat[$g_3 = \Sigma^{-1/2}_\nu v_3$.]{\includegraphics[align=c, width=\curwidth]{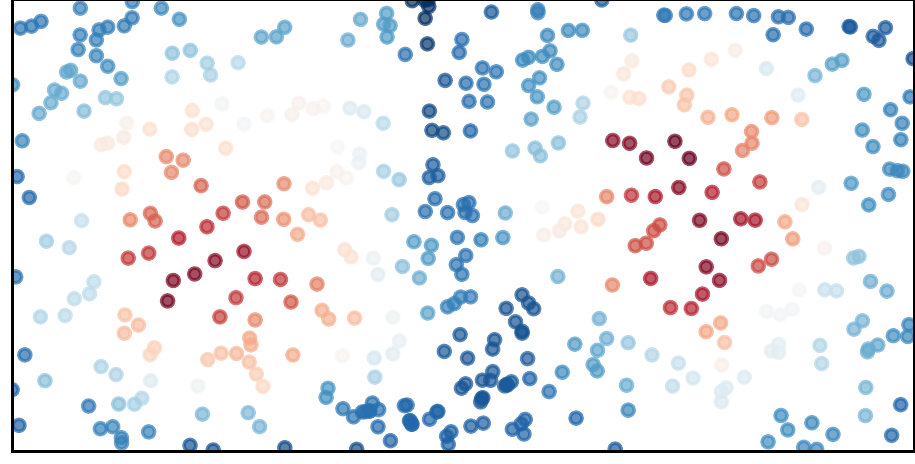}}
	\,
	\subfloat{\includegraphics[align=c, width=\curbarwidth]{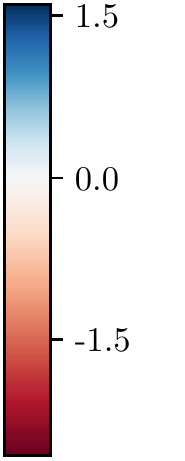}}
	\addtocounter{subfigure}{-1}
	\caption{First and second partition vectors for double gyre time series.}
	\label{fig:partition_dg}
\end{figure}

\section{Conclusions} \label{sec:conclusions}
This is the first paper that merges the theories of Frobenius--Perron operators and regularized optimal transport. We have elaborated how regularized (and possibly unbalanced) OT can be used to compute coherent sets if all we know about the dynamics of a moving particle system or a continuous quantity of moving mass is a pair of measures that constitute a preimage-image pair under the dynamical evolution.
We have also shown that the theory of (regularized/unbalanced) optimal transport is fitting well to the concept of coherent sets.
Moreover, it has natural dynamical interpretations as the regularization parameter $\varepsilon\to 0$ or the number of data points $n\to \infty$, see Proposition~\ref{prop:conv-reg} and Proposition~\ref{prop:schroedinger}, respectively.

In four numerical examples we have shown how the method performs.
These examples underline the initial suspicion that without further structural ``aid'' or dynamic information the knowledge of the one-step evolution of a single measure is not sufficient to identify coherent sets correctly.
It is necessary to incorporate additional dynamical information into the analysis.
Thus, a topic interesting to address is OT of multiple measure pairs (multiple steps of evolving one measure, or one step of evolving multiple measures), such as in segmentation of vector- and manifold-valued images,
see, e.g., \citet{CTA2018,FLS2016,FLS2017,KMDL2019,thorpe17}.
Further, working with discrete OT calls for a consistency result when approximating ground truth measures with atomic ones, as it has been given for the case of (static) spectral clustering in \citet{trillos18} using transportation distances between functions as provided in the aforementioned references.
Finally, the efficacy of unbalanced OT suggests to model certain scenarios as \emph{open dynamical systems}.

\section*{Acknowledgments}
Funding through the German Research Foundation (DFG) with\-in the project STE 571/16-1 is gratefully acknowledged by GS,  through grant CRC 1114 ``Scaling Cascades in Complex Systems'', Project Number 235221301, projects A01 ``Coupling a multiscale stochastic precipitation model to large scale atmospheric flow dynamics'' by PK\ and B07 ``Selfsimilar structures in turbulent flows and the construction of LES closures'' by JvL.
The authors want to thank Henning Rust, Institute for Meteorology at FU Berlin, for his advice regarding the data of the example from Section~\ref{subsec:precipitation}.


\end{document}